\documentclass[a4paper, 12pt]{amsart}

\usepackage{amssymb, color, hyperref}
\usepackage[left=2cm,right=2cm,top=2cm,bottom=2cm]{geometry}
\usepackage{tikz, epsfig}
\usepackage{pdfsync}
\usepackage[all]{xy}
\usetikzlibrary{decorations.markings}

\newtheorem{thm}{Theorem}[section]
\newtheorem{prp}[thm]{Proposition}
\newtheorem{cor}[thm]{Corollary}
\newtheorem{lem}[thm]{Lemma}

\theoremstyle{definition}

\newtheorem{ntn}[thm]{Notation}
\theoremstyle{remark}
\newtheorem{rmk}[thm]{Remark}

\numberwithin{equation}{section}
\DeclareMathOperator{\im}{im}

\newcommand{\NN}{\mathbb{N}}
\newcommand{\TT}{\mathbb{T}}
\newcommand{\ZZ}{\mathbb{Z}}

\newcommand{\Gg}{\mathcal{G}}
\newcommand{\Bb}{\mathcal{B}}
\newcommand\Ff{\mathcal{F}}
\newcommand{\Hh}{\mathcal{H}}
\newcommand\Ll{\mathcal{L}}
\newcommand{\Kk}{\mathcal{K}}

\newcommand{\Tt}{\mathcal{T}}

\newcommand{\Aut}{\operatorname{Aut}}

\newcommand{\dom}{\operatorname{dom}}
\newcommand{\id}{\operatorname{id}}

\newcommand{\lsp}{\operatorname{span}}
\newcommand{\clsp}{\overline{\lsp}}
\newcommand{\supp}{\operatorname{supp}}

\newcommand{\cs}{C^*}
\newcommand{\ginf}{G^{\infty}_\alpha}
\newcommand{\tginf}{\widetilde{G}^{\infty}_\alpha}
\newcommand{\ginfo}{(G^{\infty}_\alpha)^{(0)}}
\newcommand{\inv}{^{-1}}
\newcommand{\OO}{\mathcal{O}}
\newcommand{\go}{G^{(0)}}

\newcommand{\hio}{H_\infty^{(0)}}

\newcommand{\iotaG}{\iota_G}
\newcommand{\defginf}{$\ginf$ the groupoid \eqref{eq:ginf}}

\usepackage[normalem]{ulem}

\allowdisplaybreaks

\title[Kirchberg algebras from amenable principal groupoids]{Purely infinite simple $C^*$-algebras that are principal groupoid $\cs$-algebras}

\author{Jonathan H. Brown}
\address[J.H. Brown]{Department of Mathematics\\
University of Dayton\\
300 College Park Dayton\\
OH 45469-2316 USA}
\email[J.H. Brown]{jonathan.henry.brown@gmail.com}

\author{Lisa Orloff Clark}
\address[L.O. Clark]{Department of Mathematics and Statistics\\
    University of Otago\\
    PO Box 56\\
    Dunedin 9054\\
    New Zealand}
\email[L.O. Clark]{lclark@maths.otago.ac.nz}

\author{Adam Sierakowski}
\address[A. Sierakowski]{School of Mathematics and Applied Statistics\\
University of Wollongong\\
NSW 2522\\ AUSTRALIA}
\email[A. Sierakowski]{asierako@uow.edu.au}

\author{Aidan Sims}
\address[A. Sims]{School of Mathematics and Applied Statistics\\
University of Wollongong\\
NSW 2522\\ AUSTRALIA}
\email[A. Sims]{asims@uow.edu.au}

\subjclass[2010]{46L05}
\keywords{Groupoid; Kirchberg algebra}
\thanks{This research was supported by the Australian Research Council.}
\date{\today}

\begin{document}

\begin{abstract}
From a suitable groupoid $G$, we show how to construct an amenable principal groupoid
whose $C^*$-algebra is a Kirchberg algebra which is $KK$-equivalent to $C^*(G)$. Using
this construction, we show by example that many UCT Kirchberg algebras can be realised as
the $C^*$-algebras of amenable principal groupoids.
\end{abstract}

\maketitle

\section{Introduction}

Concrete models of Kirchberg algebras have proved extremely useful in classification
theory---see for example \cite{EllSie, RSS, spielberg}.  In this paper, we develop a
technique for realising many Kirchberg algebras as the $\cs$-algebras of amenable
principal groupoids.

A priori, it is not clear that any Kirchberg algebra should admit such a model.
Algebraically, a principal groupoid is just an equivalence relation; so in spirit at
least, the $C^*$-algebras of amenable principal groupoids are akin to matrix algebras.
Also, most of the existing groupoid models for Kirchberg algebras are based on graphs and
their analogues \cite{Katsura, RSS, spielberg, Szymanski}, and are not principal. But
recent work of the third-named author with R{\o}rdam shows that there are indeed examples
of amenable principal groupoids whose $C^*$-algebras are Kirchberg algebras:
\cite[Theorem~6.11]{RorSie} shows that every nonamenable exact group admits a free and
amenable action on the Cantor set for which the associated crossed-product is a Kirchberg
algebra. Because the action is free, the corresponding transformation group groupoid is
principal, and because the action is amenable, the corresponding groupoid is also
amenable. Subsequent work of the third-named author with Elliott \cite{EllSie} shows that
more than one Kirchberg algebra is achievable via this construction, but a complete range
result has yet to be established. Suzuki also shows in \cite{Suzuki:JraM14} that many
Kirchberg algebras can be constructed as crossed products arising from Cantor minimal
systems.

Here we investigate a range of Kirchberg algebras that can be modelled using amenable
principal groupoids.  Our approach is as follows: We start out with a principal groupoid
$G$ such that $\cs(G)$ is simple and the $K$-theory of $\cs(G)$ is known.  Then, given an
appropriate automorphism $\alpha$ of $G$, we construct a \emph{twisted product groupoid}
$\ginf$ and prove that $\ginf$ is principal and  amenable when $G$ is, and that
$\cs(\ginf)$ is a Kirchberg algebra. We then show that $\cs(G)$ embeds into $\cs(\ginf)$
via a homomorphism that induces a $KK$-equivalence, and in particular induces an
isomorphism in $K$-theory.

By applying the Kirchberg--Phillips classification theorem, we therefore reduce the
question of which Kirchberg algebras have amenable principal groupoid models to the
question of which pairs of $K$-groups can be obtained from simple $C^*$-algebras of
amenable principal groupoids $G$ that admit a suitable automorphism $\alpha$ (see
Theorem~\ref{thm:main}). By showing that every simple AF algebra, and every simple
A$\TT$-algebra arising from a rank-2 Bratteli diagram as in \cite{PRRS}, admits such a
model, we show that a very broad range of Kirchberg algebras can be modelled by amenable
principal groupoids. In fact, all of these examples can be modelled by amenable principal
groupoids that are ample; that is, \'etale groupoids that have a basis of compact open
bisections.

\smallskip

We outline the necessary background in Section~\ref{sec:prelims}. In
Section~\ref{sec:ginf}, we show how to form a twisted product $H \times_{c,\alpha} G$ of
groupoids $G$ and $H$ given a continuous cocycle $c : H \to \ZZ$ and an automorphism
$\alpha$ of $G$. We then investigate how properties of $G$ and $H$ relate to properties
of $H \times_{c,\alpha} G$. Most of our results are about what happens when $H$ is the
standard groupoid model $H_\infty$ for the Cuntz algebra $\OO_\infty$ and $c : H_\infty
\to \ZZ$ is the canonical cocycle. In this case we write $\ginf$ for $H_\infty
\times_{c,\alpha} G$.

In Section~\ref{sec:toepalgebra}, borrowing an idea from \cite{RenRobSim}, we show that
$\cs(\ginf)$ is isomorphic to the Toeplitz algebra $\Tt_X$ of a $C^*$-correspondence $X$
over $\cs(G)$. As a set, the groupoid $\ginf$ is just the cartesian product $H_\infty
\times G$. We show that the map $f \mapsto 1_{\hio} \times f$ from $C_c(G)$ to
$C_c(\ginf)$ extends to an embedding $\iotaG : \cs(G) \to \cs(\ginf)$. We then show that
the isomorphism $\cs(\ginf) \cong \Tt_X$ intertwines $\iotaG$ and the canonical inclusion
$i_{\cs(G)} : \cs(G) \to \Tt_X$. Combining this with Pimsner's results \cite{Pimsner}, we
deduce that the $KK$-class of $\iotaG$ is a $KK$-equivalence, and in particular,
$K_*(\iotaG)$ is an isomorphism. Though we don't need it for our later results, we also
prove that $\Tt_X$---and hence also $\cs(\ginf)$---coincides with $\OO_X$.

In Section~\ref{sec:structure}, we investigate the structure of $\ginf$. In particular we
identify conditions on $G$ that ensure  that $\cs(\ginf)$ is a Kirchberg algebra. We
summarise our results about $\ginf$ in our main theorem, Theorem~\ref{thm:main}. We
finish Section~\ref{sec:structure} with a slight modification to our construction of
$\ginf$ that shows that the class of Kirchberg algebras achievable via our construction
is closed under stabilisation.

In the last two sections, we provide two classes of examples. In Section~\ref{sec:1-graph
examples}, we show that the graph groupoids $G$ of suitable Bratteli diagrams for simple
AF algebras admit automorphisms satisfying our hypotheses. We deduce that for every
simple dimension group $D \not= \ZZ$ the stable Kirchberg algebra with $K$-theory $(D,
\{0\})$ can be realised as the $C^*$-algebra of an amenable principal groupoid. We show
further that for every element $d$ of the positive cone of $D$, the unital Kirchberg
algebra with the same $K$-theory and with the class of the unit equal to $d$ can also be
constructed in this way. In Section~\ref{sec:2-graph examples}, we apply a similar
analysis to the groupoids associated to the rank-2 Bratteli diagrams of \cite{PRRS}. Thus
for a large class of pairs $(D_0, D_1)$ of simple dimension groups with $D_0 \not=
\ZZ$ and $D_1$ a subgroup of $D_0$ for which the quotient is entirely torsional, the
stable Kirchberg algebra with $K$-theory $(D_0, D_1)$ can be realised as the
$C^*$-algebra of an amenable principal groupoid, and for many possible order units $d$
for $D_0$, the unital Kirchberg algebra with $K$-theoretic data $(D_0, d, D_1)$ can too.

\section{Preliminaries}\label{sec:prelims}

In this section, we collect the background needed for the remainder of the paper.

\subsection{Groupoids and \texorpdfstring{$C^*$}{C*}-algebras}

A groupoid is a small category $G$ with inverses. We denote the collection of identity
morphisms of $G$ by $\go$, and identify it with the object set, so that the domain and
codomain maps become maps $s,r : G \to \go$ such that $r(g)g = g = gs(g)$ for all $g \in
G$. We are interested in topological groupoids: groupoids endowed with a topology under
which inversion from $G$ to $G$ and composition from $\{(g,h) \in G \times G : s(g) =
r(h)\}$ (under the relative topology) to $G$ are continuous. By an \emph{automorphism
$\alpha$} of a topological groupoid $G$ we mean a structure-preserving homeomorphism
$\alpha\colon G\to G$. If $\alpha$ preserves the algebraic, but not necessarily the
topological, structure, we will call it an \emph{algebraic} automorphism.

A groupoid is said to be \emph{\'etale} if the range map (equivalently the source map) is
a local homeomorphism. An open set $U\subseteq G$ on which $r$ and $s$ are both homeomorphisms
is called an \emph{open bisection}; so an  \'etale groupoid is a groupoid that has a basis
consisting of open bisections. A groupoid is said to be \emph{ample} if its topology
admits a base of compact open bisections. An ample groupoid is necessarily \'etale.

We think of a groupoid $G$ as determining a partially defined action on its unit space:
for $g\in G$ and $u\in \go$, $g$ can act on $u$ if $s(g) = u$, and then $g \cdot u =
r(g)$. We write $[u]$ for the orbit $\{g\cdot u : s(g) = u\}$ of $u$ under $G$. We will
use the standard notation in the literature where $\{g \in G : s(g) = u\}$ is denoted
$G_u$, and $\{g \in G : r(g) = u\}$ is denoted $G^u$. So
\[
    [u] = r(G_u) = s(G^u).
\]
The groupoid $G$ is said to be \emph{minimal} if $\overline{[u]} = \go$ for every $u \in
\go$. It is said to be \emph{principal} if the map $g \mapsto (r(g),s(g))$ is injective.
Equivalently, $G$ is principal if, for every $u \in \go$,  $\{g \in G : s(g) = u = r(g)\}
= \{u\}$.

Each locally compact, Hausdorff, \'etale groupoid has two associated $C^*$-algebras, which
were introduced in \cite{Renault}: the full and reduced algebras $C^*(G)$ and
$C^*_r(G)$. Both are completions of the convolution $^*$-algebra $C_c(G)$ in which
\[
(\xi * \eta)(g) = \sum_{hk = g} \xi(h)\eta(k) \quad \text{ and } \quad
\xi^*(g) = \overline{\xi(g^{-1})}
\]
for $\xi,\eta \in C_c(G)$. The full $C^*$-algebra is the completion with respect to a
suitable universal norm. The reduced $C^*$-algebra is described as follows. For each $u
\in \go$, there is a representation $R_u : C_c(G) \to \Bb(\ell^2(G_u))$ such that, for
$\xi \in C_c(G)$ and $g \in G_u$,
\begin{equation}\label{eqn2.1.a}
R_u(\xi)\delta_g = \sum_{h \in G_{r(g)}}
\xi(h)\delta_{hg}.
\end{equation}
The reduced $C^*$-algebra is the completion of the image of $C_c(G)$
under $\bigoplus_{u \in \go} R_u$.

There are various notions of amenability for groupoids, but we won't give the formal
definitions of here. They all imply that $C^*(G)$ and $C^*_r(G)$ coincide. For our
purposes, it suffices to recall that when $G$ is locally compact, Hausdorff and \'etale,
\cite[Corollary~6.2.14(ii)]{A-DR} (also \cite[Theorem~5.6.18]{BO}) shows that $G$ is
amenable if and only if $C^*_r(G)$ is nuclear --- and then $C^*(G)$ automatically
coincides with $C^*_r(G)$.

\subsection{Graph groupoids}\label{subsec:graphgpd}

A key example for us will be the graph groupoids  introduced in \cite{KPRR}, and
subsequently extended in \cite{Pat02} to allow for non-row-finite graphs. For the
majority of the paper, we will be interested only in the groupoid of the graph $E$ with
one vertex $v$ and infinitely many edges $\{e_i:i \in \NN\}$; but we will need greater generality in
Section~\ref{sec:1-graph examples}.

A directed graph $E$ consists of two countable sets $E^0$ and $E^1$ and two maps $r, s :
E^1 \to E^0$. We will assume that our graphs $E$ have no sources in the sense that
$r^{-1}(v) \not= \emptyset$ for all $v \in E^0$; but we will allow them to be
non-row-finite in the sense that $r^{-1}(v)$ may be infinite.

A \emph{path} in a directed graph $E$ is a vertex, or a word $\mu = \mu_1 \dots \mu_n$
over the alphabet $E^1$ with the property that $s(\mu_i) = r(\mu_{i+1})$ for all $i$. We
write $s(\mu)$ for $s(\mu_n)$ and $r(\mu)$ for $r(\mu_1)$. If $\mu$ is a vertex $v$, then
we define $s(\mu) = r(\mu) = v$. An \emph{infinite path} is a right-infinite word $x =
x_1 x_2 \dots$ over the alphabet $E^1$ satisfying $s(x_i) = r(x_{i+1})$ for all $i$. The
space of all finite paths is denoted $E^*$ and the space of all infinite paths is denoted
$E^\infty$. We write $|\mu|$ for the number of edges in a path $\mu$, with the convention
that $|v| = 0$ for $v \in E^0$ and $|\mu| = \infty$ when $\mu$ is an infinite path. For a
path $\mu$ (either finite or infinite) and an integer $n \le |\mu|$, we write $\mu(0,n)
:= \mu_1 \dots \mu_n \in E^*$. We write $P_E$, or just $P$ for the collection
\[
E^\infty \sqcup\{\mu \in E^* : r^{-1}(s(\mu))\text{ is infinite}\}.
\]
For each $\mu \in E^*$ we define
\[
    Z(\mu) := \{\alpha \in P : |\alpha| \ge |\mu|\text{ and }\alpha(0, |\mu|) =\mu\}
\]
and then for each finite set $F \subseteq r^{-1}(s(\mu))$ we define
\[\textstyle
Z(\mu \setminus F) := Z(\mu) \setminus \bigcup_{e \in F} Z(\mu e).
\]
The collection of all such $Z(\mu \setminus F)$ form a base of compact sets that generate
a locally compact Hausdorff topology on $P$. In particular, for $\mu \in E^*$, a base of
neighbourhoods of $\mu$ is the collection $\{Z(\mu \setminus F) : F \subseteq
r^{-1}(s(\mu))\text{ is finite}\}$, and for $\mu \in E^\infty$ a base of neighbourhoods
of $\mu$ is $\{Z(x(0,n)) : n \in \NN\}$.

We borrow some notation from the $k$-graph literature. For $\mu \in E^*$ and $F \subseteq
P$, we write $\mu F$ for $\{\mu\lambda : \lambda \in F, r(\lambda) = s(\mu)\}$, similarly
for $F \subseteq E^*$, we write $F\mu = \{\lambda\mu : \lambda \in F, s(\lambda) =
r(\mu)\}$. So, for example, $\mu P$ is equal to $Z(\mu)$, and $vE^1 = r^{-1}(v)$.

For $n \in \NN$, there is map $\sigma^n : \{\mu \in P : |\mu| \ge n\} \to P$ given by
$\sigma^{n}(\mu) = s(\mu)$ for $\mu \in E^n$ and $\sigma^n(\mu) = \mu_{n+1} \cdots
\mu_{|\mu|}$ for $|\mu| > n$. Observe that $\sigma^n$ restricts to a homeomorphism on
$Z(\mu)$ whenever $|\mu| \ge n$. Also observe that the domain of $\sigma^{m+n}$ is
$(\sigma^{m})^{-1}(\dom\sigma^n)$ and on this domain, $\sigma^{m+n} = \sigma^n \circ
\sigma^m$.

As a set, the \emph{graph groupoid} $G_E$ of $E$ is equal to
\begin{equation}\label{eqn.2.2}
    G_E = \{(x,m-n,y) : x, y \in P \text{ and } \sigma^m(x) = \sigma^n(y)\}.
\end{equation}
Its unit space is $\{(x,0,x) : x \in P\}$, which we identify with $P$, and multiplication
and inverse are $(x,m,y)(y,n,z) = (x, m+n, z)$ and $(x,m,y)\inv = (y, -m, x)$. The
topology has basic compact open sets
\[
Z((\alpha,\beta) \setminus F) = \{(\alpha x,
|\alpha| - |\beta|, \beta x) : x \in Z(s(\alpha) \setminus F)\}
\]
indexed by pairs
$(\alpha,\beta) \in E^*$ with $s(\alpha) = s(\beta)$ and finite sets $F \subseteq
s(\alpha)E^1$.
\begin{rmk}\label{rem.2.1}
Under this topology $G_E$ is an amenable, second-countable, locally compact, Hausdorff,
ample groupoid \cite{Pat02}. If $\cs(E)$ denotes the graph algebra of $E$, then there is
an isomorphism $\cs(E) \cong \cs(G_E)$ that carries each generator $s_e$ to the
characteristic function $1_{Z(e,s(e))}$.
\end{rmk}

\subsection{\texorpdfstring{$C^*$}{C*}-correspondences and \texorpdfstring{$C^*$}{C*}-algebras}

Following \cite{FR99}, given a $C^*$-algebra $A$, we say that a right $A$-module $X$ is a
\emph{right inner-product module} if it is endowed with a map
$\langle\cdot,\cdot\rangle_A : X \times X \to A$ that is linear and $A$-linear in the
second variable, satisfies $\langle x, y\rangle_A^* = \langle y,x\rangle_A$, and
satisfies $\langle x, x\rangle_A \ge 0$ for all $x$, with equality only for $x = 0$. The
formula $\|x\| := \|\langle x, x\rangle_A\|^{1/2}$ defines a norm on $X$, and $X$ is a
\emph{right Hilbert $A$-module} if it is complete in this norm. An operator $T$ on $X$ is
\emph{adjointable} if there is an operator $T^*$ on $X$ such that $\langle Tx,y\rangle_A
= \langle x, T^* y\rangle_A$ for all $x,y$; $T$ is then linear, bounded and $A$-linear
and $T^*$ is unique. The space $\Ll(X)$ of all adjointable operators is a $C^*$-algebra
under the operator norm. For $x,y \in X$, the formula $\Theta_{x,y}(z) := x \cdot \langle
y,z\rangle_A$ determines an adjointable operator with adjoint $\Theta_{x,y}^* =
\Theta_{y,x}$. The set $\clsp\{\Theta_{x,y} : x,y \in X\}$ is an ideal of $\Ll(X)$
denoted $\Kk(X)$. For more details, see for example \cite[Section~2.2]{tfb}.

A right Hilbert $A$-module $X$ over a $C^*$-algebra $A$ becomes a
\emph{$C^*$-correspondence over $A$} when endowed with a homomorphism $\phi : A \to
\Ll(X)$, which we then regard as defining a left action of $A$ on $X$: $a \cdot x :=
\phi(a)x$. A \emph{Toeplitz representation}, or just a \emph{representation}, of $X$ in a
$C^*$-algebra $B$ is a pair $(\psi,\pi)$ consisting of a linear map $\psi : X \to B$ and
a homomorphism $\pi : A \to B$ such that
\[
\pi(a)\psi(x) = \psi(a \cdot x), \quad \psi(x)\pi(a) = \psi(x\cdot a) \quad \text{ and } \quad \psi(x)^*\psi(y) = \pi(\langle x, y\rangle_A)
\]
for all $a\in A$ and $x,y \in X$.

Each $C^*$-correspondence $X$ over $A$ has two associated $C^*$-algebras: the
\emph{Toeplitz algebra} $\Tt_X$ and the \emph{Cuntz--Pimsner algebra} $\OO_X$
\cite{Pimsner}. The Toeplitz algebra $\Tt_X$ is the universal $C^*$-algebra generated by
a Toeplitz representation $(i_X, i_A)$ of $X$. Every representation of $X$ induces a
homomorphism $\psi^{(1)} : \Kk(X) \to B$ satisfying $\psi^{(1)}(\Theta_{x,y}) = \psi(x)
\psi(y)^*$. If the homomorphism $\phi : A \to \Ll(X)$ implementing the left action is
injective, then we say that the representation $(\psi,\pi)$ is \emph{Cuntz--Pimsner
covariant} if
\[
\psi^{(1)}(\phi(a)) = \pi(a)\quad\text{ whenever $\phi(a) \in \Kk(X)$.}
\]
The Cuntz--Pimsner algebra $\OO_X$ of $X$ is the quotient of $\Tt_X$ that is
universal for Cuntz--Pimsner covariant representations of $X$.

\section{A twisted product of groupoids}\label{sec:ginf}

Let $H$ be a topological groupoid endowed with a cocyle $c:H \to \ZZ$ (a cocycle is a map
satisfying $c(gh)=c(g)+c(h)$ whenever $s(g)=r(h)$) which is continuous. Let $G$ be
another topological groupoid, and suppose that $\alpha : G \to G$ is an automorphism of
$G$. In this section, we describe how to construct a \emph{twisted product groupoid} $H
\times_{c,\alpha} G$. Our construction makes sense for any $H$ and $G$, but we will later
be interested primarily in the situation where $H$ is the groupoid associated to
$\OO_{\infty}$ (see~\eqref{eq:hinf}).

\begin{lem}\label{lem:twisted product}
Let $G$ and $H$ be groupoids. Suppose that $c : H \to \ZZ$ is a cocycle, and that $\alpha : G
\to G$ is an algebraic automorphism. Let
\[
    H \times_{c,\alpha} G := H \times G,
\]
and $(H \times_{c,\alpha} G)^{(0)} := H^{(0)} \times \go$. Define $r, s : H
\times_{c,\alpha} G \to (H \times_{c,\alpha} G)^{(0)}$ by $r(h,g) := \big(r(h),
r(g)\big)$ and $s(h,g) := \big(s(h), \alpha^{c(h)}(s(g))\big)$. For $(h_1, g_1), (h_2,
g_2) \in H \times_{c,\alpha} G$ with $s(h_1, g_1) = r(h_2, g_2)$, define
\[
(h_1,g_1)(h_2,g_2):=(h_1h_2, g_1\alpha^{-c(h_1)}(g_2)) \quad\text{ and }\quad	
    (h,g)\inv:=(h\inv,\alpha^{c(h)}(g\inv)).
\]
Under these structure maps, the set $H \times_{c,\alpha} G$ is a groupoid.
\end{lem}
\begin{proof}
Clearly $r((h_1, g_1)(h_2,g_2)) = r(h_1, g_1)$. Since $\ZZ$ is abelian and $\alpha$ is an
automorphism, we have $\alpha^{c(h_1h_2)}(s(\alpha^{-c(h_1)}(g_2))) =
\alpha^{c(h_2)}\circ\alpha^{c(h_1)}(\alpha^{-c(h_1)}s(g_2)) = \alpha^{c(h_2)}(s(g_2))$,
and so
\[
s((h_1, g_1)(h_2,g_2))
    = \big(s(h_1h_2), \alpha^{c(h_1h_2)}(s(\alpha^{-c(h_1)}(g_2)))\big)
    = \big(s(h_1h_2), \alpha^{c(h_2)}(s(g_2))\big)
    = s(h_2, g_2).
\]
Now if $(h_1, g_1), (h_2, g_2)$ and $(h_3,g_3)$ are composable, then
\begin{align*}
\big((h_1,g_1)(h_2,g_2)\big)(h_3,g_3)
    &= (h_1h_2, g_1\alpha^{-c(h_1)}(g_2))(h_3,g_3)
    = (h_1h_2h_3, g_1 \alpha^{-c(h_1)}(g_2)\alpha^{-c(h_1h_2)}(g_3)) \\
    &= (h_1, g_1)\big(h_2h_3, g_1\alpha^{-c(h_1)}(g_2\alpha^{-c(h_2)}(g_3))\big)
    = (h_1,g_1)\big((h_2,g_2)(h_3,g_3)\big).
\end{align*}
So multiplication preserves ranges and sources and is associative. We have
\[
(h,g)s(h,g)
    = (h,g)(s(h), \alpha^{c(h)}(s(g)))
    = (h s(h), g\alpha^{-c(h)}(\alpha^{c(h)}(s(g))))
    = (h,g),
\]
and an even simpler calculation gives $r(h,g)(h,g) = (h,g)$. We have
\[
(h,g)^{-1}(h,g)
    = (h\inv h, \alpha^{c(h)}(g\inv)\alpha^{-c(h^{-1})}(g))
    = (s(g), \alpha^{c(h)}(s(g))) = s(h,g),
\]
and a similar calculation gives $(h,g)(h\inv,\alpha^{c(h)}(g\inv)) = r(h,g)$. Finally we have
\[
s((h,g)\inv)=s(h\inv,\alpha^{c(h)}(g\inv))=\big(s(h\inv), \alpha^{c(h\inv)}\big(s(\alpha^{c(h)}
(g\inv))\big)\big)=(s(h\inv), s(g\inv))=r(h,g),
\]
completing
the proof.
\end{proof}

We now show that if $H$ and $G$ are both locally compact, Hausdorff, \'etale groupoids,
then so is $H \times_{c,\alpha} G$.

\begin{lem}
\label{lem:etale} Let $H$ and $G$ be locally compact, Hausdorff, \'etale groupoids, $c:H
\to \ZZ$ be a continuous cocycle and $\alpha:G \to G$ an automorphism. Then the groupoid
$H \times_{c,\alpha} G$ is locally compact, Hausdorff and \'etale when endowed with the
product topology. If $G$ and $H$ are second countable, then so is $H \times_{c,\alpha}G$.
If $G$ and $H$ are ample, then so is $H \times_{c,\alpha}G$.
\end{lem}
\begin{proof}
Since $H$ and $G$ are locally compact and Hausdorff, so is the product $H \times G$.
Since $H \times_{c,\alpha} G$ is $H \times G$ as a topological space, we deduce that $H
\times_{c,\alpha} G$ is locally compact and Hausdorff. Composition and inversion in $H
\times_{c,\alpha} G$ are continuous because each $\alpha^n$ is continuous and composition
and inversion in each of $H$ and $G$ are continuous. Since the range map on $H
\times_{c,\alpha} G$ coincides with that on the \'etale groupoid $H \times G$, it is a
local homeomorphism, and so $H \times_{c,\alpha} G$ is \'etale too.

The statement about second countability is clear. For the final statement, observe that
if $\mathcal{U}$ and $\mathcal{V}$ are bases of compact open bisections for $H$ and $G$
respectively, then $\mathcal{U} \times \mathcal{V}$ is a base of compact open bisections
for $H \times_{c,\alpha} G$.
\end{proof}

\subsection{The twisted product groupoid\texorpdfstring{ $\ginf$}{}}

From now on, we restrict our attention to the situation where $G$ is locally compact,
Hausdorff and \'etale and $H$ is the groupoid associated to the Cuntz algebra
$\OO_{\infty}$, which we make precise below.

Let $E$ be the graph with one vertex $v$ and infinitely many edges $\{e_i \colon i\in
\NN\}$. Thus, the graph $\cs$-algebra $\cs(E)$ is canonically isomorphic to the Cuntz
algebra $\OO_{\infty}$ (see for example \cite[page~42]{RS}).  We denote by $H_\infty$ the
groupoid $G_E$. Since $C^*(H_\infty)\cong \OO_\infty$, we call $H_\infty$ \emph{the
groupoid associated to the Cuntz algebra $\OO_{\infty}$}. As discussed in
Section~\ref{sec:prelims}, writing $P:=E^*\cup E^\infty$, the groupoid $H_\infty$
consists of ordered triples
\begin{equation}\label{eq:hinf}
H_\infty = \{(\alpha x,|\alpha|-|\beta|, \beta x): x\in P, \alpha,\beta\in
			E^*r(x)\}
    \subseteq P \times \ZZ \times P,
\end{equation}
with operations $(x, m, y)(y, n, z)=(x, m+n, z)$ and $(x, m, y)\inv= (y, -m, x)$. We
identify $\hio$ with $P$ via $(x,0,x) \mapsto x$. The topology on $H_\infty$ described in
Section~\ref{subsec:graphgpd} is generated by the sets
\[
    Z((\alpha,\beta) \setminus F) = \{(\alpha x,|\alpha|-|\beta|, \beta x) : x \in P, x_1 \not\in F\}
\]
indexed by pairs $\alpha,\beta \in E^*$ and finite subsets $F \subseteq \{e_i : i \in
\NN\}$. We define $Z(\alpha,\beta) := Z((\alpha,\beta)\setminus \emptyset)$. For $\alpha
\in P$ we have
\[
    \alpha{P} = Z(\alpha) = Z(\alpha,\alpha) = \{\alpha x : x \in P\}.
\]
With this structure $H_\infty$ is a second-countable, amenable, locally compact,
Hausdorff, ample groupoid; the sets $Z((\alpha,\beta) \setminus F)$ are compact open
bisections.

\begin{lem}\label{lem:handy}
For every nonempty open set $W \subseteq \hio$ there exists $\lambda
\in E^*$ such that $Z(\lambda) \subseteq W$.
\end{lem}
\begin{proof}
Fix any $u \in W$. If $u \in E^\infty$ then the sets $\{Z(u(0,n)) : n \in \NN\}$ are a
neighbourhood base at $u$, so we can take $\lambda = u(0,n)$ for large $n$. If $u \in
E^*$, then the sets $Z(u \setminus F)$, where $F\subseteq s(u)E^1$ is finite, form a
neighbourhood base at $u$. Hence there exists finite $F$ such that $Z(u \setminus
F)\subseteq W$. Let $n=\max\{j\colon e_j\in F\}+1$. Then $\lambda = u e_n$ satisfies
$Z(\lambda)\subseteq Z(u \setminus F)\subseteq W$.
\end{proof}

\begin{ntn}
Let $G$ be any locally compact, Hausdorff, \'etale groupoid and let $\alpha$
be an automorphism of $G$. We let $\ginf$ denote the twisted product of $G$ and
$H_\infty$ with its canonical cocycle, i.e.,
\begin{equation}\label{eq:ginf}
\ginf := H_\infty \times_{c, \alpha} G,
\end{equation}
as in Lemma~\ref{lem:twisted product}, where $c : H_\infty \to \ZZ$ is given by
$c(x,m,y) = m$.
\end{ntn}

\begin{lem} \label{lem:amenable}
Let $G$ be a locally compact, Hausdorff, \'etale groupoid, $\alpha$ an automorphism of
$G$, and \defginf. Then $\ginf$ is amenable if and only if $G$ is amenable; and then
$\cs(\ginf)\cong \cs_r(\ginf)$ is nuclear and in the UCT class.
\end{lem}

\begin{proof}
First suppose that $G$ is amenable. Define a map $\tilde{c}:\ginf \to \ZZ$  by
$\tilde{c}(h,g) = c(h)$. Then $\tilde{c}$ is a cocycle on $\ginf$. Using
\cite[Proposition~9.3]{spielberg} (see also \cite[Corollary~4.5]{RenaultWilliams:xx15}),
we show that $\ginf$ is amenable by showing that $\tilde{c}\inv(0)$ is amenable. The
formula for multiplication in $G$ restricts to coordinatewise multiplication on the
subgroupoid $\tilde{c}\inv(0)$, and so $\tilde{c}\inv(0) \cong c\inv(0) \times G$ as
topological groupoids. Since products of amenable groupoids are again amenable, and since
$G$ is amenable by assumption, it therefore suffices to show that $c\inv(0) \subseteq
H_\infty$ is amenable. As above, let $E$ be the graph with one vertex and infinitely many
loops, so that $H_\infty$ is the graph groupoid of $E$. For each $n \in \NN$, Paterson
shows \cite[Proposition~4.1]{Pat02} that the subgroupoid $R_n := \bigcup^n_{i=1}\{(\alpha
x, 0, \beta x) : \alpha,\beta \in E^i, x \in E^* \cup E^\infty\}$ of $H_\infty$ is
amenable. He then shows (see the two paragraphs following \cite[Proposition~4.1]{Pat02})
that $R := \bigcup^\infty_{n=1} R_n$ is also amenable. This $R$ is precisely $c\inv(0)$,
so $c\inv(0)$, and hence $\tilde{c}\inv(0)$ is amenable.

Now suppose that $\ginf$ is amenable. Fix $x \in \hio$. Then the subgroupoid $\{x\}
\times G$ of $\ginf$ is closed, and therefore locally closed, so
\cite[Proposition~5.1.1]{A-DR} shows that $\{x\} \times G$ is amenable. Since $G$ is
canonically isomorphic to $\{x\} \times G$, it follows that $G$ is amenable.

By Lemma \ref{lem:etale}, $\ginf$ is locally compact, Hausdorff and \'etale and so $C^*_r(\ginf)$ is defined. Suppose that
$\ginf$ is amenable. Then $\cs(\ginf) \cong \cs_r(\ginf)$ by
\cite[Proposition~6.1.18]{A-DR}. It is nuclear by \cite[Corollary~6.2.14(i)]{A-DR}, and
belongs to the UCT class by \cite[Lemma~3.3 and Proposition~10.7]{Tu}.
\end{proof}

\section{Realising \texorpdfstring{$\cs(\ginf)$}{the twisted-product groupoid C*-algebra} as a Toeplitz algebra}
\label{sec:toepalgebra}

Throughout this section we assume that $G$ is amenable.  Under this hypothesis we will show that $\cs(\ginf)$ can be realised as the Toeplitz algebra of a
correspondence over $\cs(G)$. Provided that $\cs(G)$ is separable, it will then follow
from \cite[Theorem~4.4]{Pimsner} that $\cs(\ginf)$ is $KK$-equivalent to $\cs(G)$.

We continue to write $H_\infty$ for the groupoid~\eqref{eq:hinf} associated
to the Cuntz algebra $\OO_{\infty}$. For $\phi \in C_c(H_\infty)$ and $f \in C_c(G)$ we write
$\phi \times f$ for the function in $C_c(\ginf)$ defined by
\[
    \phi \times f (h,g) = \phi(h) f(g).
\]

\begin{lem}\label{lem:embedding}
Let $G$ be an amenable, locally compact, Hausdorff, \'etale groupoid, $\alpha$ an
automorphism of $G$, and \defginf. Then the map $\iotaG: f\mapsto 1_{\hio} \times f$ from
$C_c(G)$ to $C_c(\ginf)$ extends to an embedding $\iotaG : C^*(G) \to C^*(\ginf)$.
\end{lem}
\begin{proof}
Since $\hio$ is compact and open in $H_\infty$, the characteristic function $1_{\hio}$ belongs to
$C_c(H_\infty)$, and so $\iota$ makes sense. Since $c$ is zero on $\hio$, the operations
on $\hio \times G$ in $\ginf$ agree with those in $H_\infty \times G$. Since $f \mapsto
1_{\hio} \times f$ is multiplicative and star-preserving from $C_c(G)$ to $C_c(H_\infty
\times G)$, it follows that $\iota$ is a homomorphism.

Fix $(u,x) \in \hio \times \go$ and $h\in (H_\infty)_u$. Define $y:=\alpha^{-c(h)}(x)\in
\go$ and let $P_h$ denote the projection of $\ell^2((\ginf)_{(u,x)})$ onto the subspace
$\ell^2(\{h\}\times G_y)$. Define $V_h : \ell^2(\{h\}\times G_y)\to \ell^2(G_y)$ by $V_h
\delta_{(h,\gamma)} = \delta_\gamma$. Fix $(s,\lambda)\in(\ginf)_{(u,x)}$. We have
$V_h\delta_{(s,\lambda)} = V_hP_h\delta_{(s,\lambda)}$, which is equal to
$\delta_{\lambda}$ if $s=h$ and zero otherwise. So~\eqref{eqn2.1.a} gives
\[
V^*_hR_y(f)V_h \delta_{(s,\lambda)}
    = 1_{\{h\}}(s)V^*_hR_y(f)\delta_{\lambda}
    = 1_{\{h\}}(s)V^*_h\sum_{\mu \in G_{r(\lambda)}}f(\mu)\delta_{\mu\lambda}
    = 1_{\{h\}}(s)\sum_{\mu \in G_{r(\lambda)}}f(\mu)\delta_{(h,\mu\lambda)}.
\]
Moreover, by \eqref{eqn2.1.a} again, $R_{(u,x)}(\iotaG(f))
\delta_{(s,\lambda)}=\sum_{(g,\mu) \in
(\ginf)_{r(s,\lambda)}}\iotaG(f)(g,\mu)\delta_{(g,\mu)(s,\lambda)}$ with nonzero terms in
the sum only when $g=r(s)$, in which case $\iotaG(f)(g,\mu)\delta_{(g,\mu)(r,\lambda)}
=f(\mu)\delta_{(r,\mu\lambda)}$. We get
$$P_h R_{(u,x)}(\iotaG(f)) \delta_{(s,\lambda)}
=P_h\sum_{(r(s),\mu) \in (\ginf)_{r(s,\lambda)}}f(\mu)\delta_{(s,\mu\lambda)}
=1_{\{h\}}(s)\sum_{\mu \in G_{r(\lambda)}}f(\mu)\delta_{(h,\mu\lambda)}.$$

Summing over all $h\in (H_\infty)_u$, we have $R_{(u,x)}(\iotaG(f)) = \bigoplus_h V_h^*
R_{\alpha^{-c(h)}(x)}(f) V_h$. That is, each regular representation of $\ginf$ is
equivalent to a direct sum of regular representations of $C_c(G)$, which gives
$\|R_{(u,x)}(\iotaG(f))\| \le \|f\|_{C^*_r(G)}$ for all $f$. Also, each regular
representation $R_x$ of $G$ is equivalent to a direct summand in a regular representation
of $C_c(\ginf)$, giving $\|R_x(f)\| \le \|\iotaG(f)\|_{C^*_r(\ginf)}$. So $\iotaG$ is
isometric for the reduced norms on $C_c(G)$ and $C_c(\ginf)$.
\end{proof}

Recall that $H_\infty$ is the groupoid of the graph with one vertex $v$ and infinitely
many edges $\{e_i \colon i\in \NN\}$. For $i\in \NN$, let $x_i := 1_{Z(e_i,v)} \in
C_c(H_\infty)$. Define
\[
    X:=\clsp\{x_i \times f : f\in C_c(G), i\in \NN\} \subseteq \cs(\ginf).
\]
We will show that $X$ is a $C^*$-correspondence over $\cs(G)$ and then apply
\cite[Theorem~3.1]{FR99} to show that $\cs(\ginf) \cong \Tt_X$.

\begin{lem}\label{lem:module}
Let $G$ be an amenable, locally compact, Hausdorff, \'etale groupoid, $\alpha$ an
automorphism of $G$, and \defginf. With
$$a \cdot x:=\iotaG(a)x, \qquad x \cdot a:=x\iotaG(a), \qquad \textrm{and}\qquad \langle x,y
\rangle_{\cs(G)}:=\iotaG\inv (x^*y)$$
for $x,y \in X$ and $a\in \cs(G)$, the space $X$ is a
$\cs$-correspondence over $\cs(G)$.
\end{lem}
\begin{proof}
We first show that for $x,y \in X$, we have $x^*y \in \iota_G(\cs(G))$. Since convolution
and involution are continuous and linear, it suffices to consider $x=x_i \times f$ and
$y=x_j \times f'$. We have
\begin{align}
(x_i\times f)^**(x_j \times f')(h,g)
	&=\sum_{(h_1,g_1)(h_2,g_2)=(h,g)} \overline{(x_i\times f)((h_1,g_1)^{-1})}(x_j \times f')(h_2,g_2) \label{eq:comp}\\
	&=\sum_{(h_1,g_1)(h_2,g_2)=(h,g)} \overline{(x_i\times f)(h_1^{-1},\alpha^{c(h_1)}(g_1^{-1}))}(x_j \times f')(h_2,g_2)\notag\\
    &= \sum_{(h_1,g_1)(h_2,g_2)=(h,g)} 1_{Z(v,e_i)}(h_1) \overline{\big(f\circ \alpha^{c(h_1)}\big)(g_1^{-1})} 1_{Z(e_j,v)}(h_2)f'(g_2)\notag\\
    &= \sum_{((r(h),-1,e_ir(h)),g_1)((e_js(h),1,s(h)),g_2) = (h,g)} (f\circ\alpha^{-1})^*(g_1)f'(g_2)\notag\\
    &= \delta_{i, j} 1_{\hio}(h) \sum_{g_1 \alpha(g_2) = g}  (f\circ\alpha^{-1})^*(g_1)f'(g_2)\notag\\
	&= \delta_{i, j} 1_{\hio}(h) \sum_{g_1 g_2' = g}  (f\circ\alpha^{-1})^*(g_1)(f'\circ\alpha^{-1})(g_2')\notag\\
    &= \delta_{i, j} \iotaG\big((f\circ \alpha^{-1})^**(f'\circ \alpha^{-1})\big)(h,g), \notag
\end{align}
which belongs to $\iotaG(\cs(G))$ as claimed. In particular the inner product is well
defined. Since $\iota_G$ is isometric by Lemma~\ref{lem:embedding}, The inner-product
norm on $X$ agrees with the $\cs$-norm on $\cs(\ginf)$. So $X$ is complete.

Now we check that $\iotaG(\cs(G)) X$ and $X \iotaG(\cs(G))$ are contained in $X$. For
$f,f' \in C_c(G)$ and $i \in \NN$ we have
\begin{align}
(f'\cdot (x_i\times f)) (h,g)
    &= \big(\iotaG(f')*(x_i \times f)\big) (h,g)\label{eq:comp2}\\
    &= \sum_{(h_1,g_1)(h_2,g_2)=(h,g)} (1_{\hio}\times f')(h_1,g_1)(1_{Z(e_i,v)} \times f )(h_2,g_2)\notag\\
    &= \sum_{g_1\alpha^{-c(r(h))}(g_2) = g}  f'(g_1)1_{Z(e_i,v)}(h)f(g_2)\notag\\
    &= 1_{Z(e_i,v)}(h)\sum_{g_1g_2=g} f'(g_1)f(g_2)
     = \big(x_i \times (f'*f)\big)(h,g).\notag
\end{align}
Thus by continuity of multiplication in $\cs(\ginf)$, we have $\iotaG(\cs(G)) X
\subseteq X$. Similarly,
\begin{align*}
((x_i \times f) \cdot f')(h,g)
    &= \sum_{(h_1,g_1)(h_2,g_2)=(h,g)} (1_{Z(e_i,v)} \times f )(h_1,g_1)(1_{\hio}\times f')(h_2,g_2)\\
    &= \sum_{g_1\alpha^{-1}(g_2) = g}  1_{Z(e_i,v)}(h) f(g_1)f'(g_2)\\
    &= 1_{Z(e_i,v)}(h)\sum_{g_1g'_2=g} f(g_1)f'(\alpha(g'_2))\\
    &= \big(x_i\times (f*(f'\circ \alpha))\big)(h,g).
\end{align*}
Since $\alpha \in \Aut(G)$ we have $f' \circ \alpha \in C_c(G)$. So continuity gives $X
\iotaG(\cs(G)) \subseteq X$.

The $C^*$-algebra $\cs(\ginf)$ is a correspondence over itself with actions given by
multiplication and inner-product $(a,b) \mapsto a^*b$, and we have just showed that $X
\subseteq \cs(\ginf)$ is an inner-product bimodule over $\cs(G) \subseteq \cs(\ginf)$ under
the inherited operations. Thus $X$ is a correspondence over $\cs(G)$ as required.
\end{proof}

To fix notation, recall that $\Tt_X$ is generated by a universal representation $(i_X,
i_{\cs(G)})$ of $X$. Moreover, every Toeplitz representation $(\psi,\pi)$ of $X$ in
$B(\Hh)$ induces a homomorphism $\psi\times \pi\colon \Tt_X \to B(\Hh)$ satisfying
$(\psi\times \pi)\circ i_X= \psi$ and $(\psi\times \pi)\circ i_{\cs(G)}=\pi$
(\cite[Proposition 1.3]{FR99}).

\begin{lem}\label{lem:4.3}
Let $X$ be a $C^*$-correspondence over a $C^*$-algebra $A$. Suppose that
$X=\bigoplus_{i\in\NN}X_i$ and that for each nonzero $a\in A$ the set $\{i\in\NN\colon
a\cdot X_i\neq \{0\}\}$ is infinite. Suppose that $(\psi,\pi)$ is a representation of $X$
in $B(\Hh)$ such that $\pi$ is injective. Then $\psi\times \pi$ is injective.
\end{lem}
\begin{proof}
Fix a finite set $F\subseteq\NN$. By \cite[Theorem~3.1]{FR99}, $\psi\times \pi$ is
injective provided that the compression of $\pi$ to $(\psi(\oplus_{j\in F}X_j)\Hh)^\bot$
is faithful. Fix $a\in A \setminus \{0\}$. Take $i\not\in F$ and $x\in X_i$ such that
$a\cdot x\neq 0$. Since $\pi$ is injective if follows from \cite[Remark 1.1]{FR99} that
$\psi$ is isometric. We can therefore find $h\in \Hh$ satisfying $\psi(a\cdot x)h\neq 0$.
For each $j\in F$, $y\in X_j$ and $k\in\Hh$ we have $\big(\psi(x)h \mid \psi(y)k\big)
    = \big(h \mid \psi(x)^*\psi(y)k\big)
    = \big(h \mid \pi(\langle x,y \rangle_{\cs(G)})k\big)
    = 0$,
so
\[
    0\neq \psi(a\cdot x)h=\pi(a)\psi(x)h\in \pi(a)(\psi(\oplus_{i\in F}X_i)\Hh)^\bot.
\]
We conclude that the compression of $\pi$ to $(\psi(\oplus_{j\in F}X_j)\Hh)^\bot$ is
faithful.
\end{proof}

Define $X_i:=\clsp\{x_i \times f : f\in C_c(G)\}$ for each $i\in \NN$, with the
module structure and left action inherited from $X$.

\begin{lem}\label{lem:graded}
Let $G$ be an amenable, locally compact, Hausdorff, \'etale groupoid, $\alpha$ an
automorphism of $G$, and \defginf. For each $i\in \NN$, the space $X_i$ is
$C^*$-correspondence over $\cs(G)$ and the module $X$ of Lemma~\ref{lem:module} is
isomorphic to $\bigoplus_i X_i$.
\end{lem}
\begin{proof}
That each $X_i$ is a correspondence follows from the argument of Lemma~\ref{lem:module}.
By (\ref{eq:comp}) it follows that $\langle x,y\rangle_{\cs(G)}=0$ for $x\in X_i$, $y\in
X_j$ and $i\neq j$. Using this, it is routine to check that the map from the algebraic
direct sum of the $X_i$ to $X$ that carries $(\xi_i)^\infty_{i=1}$ to $\sum^\infty_{i=1}
\xi_i$ preserves the inner product and so is isometric. Since the module actions in both
$X$ and $\bigoplus X_i$ are implemented by multiplication in $\cs(\ginf)$, the map $(x_i)
\mapsto \sum x_i$ is also a bimodule map, completing the proof.
\end{proof}

\begin{lem}\label{lem:4.5}
Let $G$ be an amenable, locally compact, Hausdorff, \'etale groupoid, $\alpha$ an
automorphism of $G$, and \defginf.  Suppose $X$ and  $X_i$  are the $C^*$-correspondences
of Lemma~\ref{lem:graded}. For each $a\in C^*(G)$ there exist $f \in C_0(\go)$ such that
\[
    \|a \cdot (x_i \times f)\|\geq \frac{\|a\|}{3}
\]
for all $i\in\NN$. In particular for $a\neq 0$, the set $\{i\in\NN\colon a\cdot X_i\neq
\{0\}\}$ is infinite.
\end{lem}
\begin{proof}
Fix $a\in \cs(G)$. Choose $a_0 \in C_c(G)$ satisfying $\|a - a_0\| \leq \|a\|/3$. Then
the set
\[
    s(\supp(a_0)) \subseteq \go
\]
is compact, so we can choose $f \in C_c(\go)$ such that $\|f\| = 1$ and
$f|_{s(\supp(a_0))} \equiv 1$, and hence $a_0 * f = a_0$. The computations
\eqref{eq:comp}~and~\eqref{eq:comp2} show that
\[
    \|a_0 \cdot (x_i \times f)\| = \|x_i \times (a_0 * f)\|= \|(a_0 * f) \circ \alpha^{-1}\|=\|a_0 \circ \alpha^{-1}\| = \|a_0\|
\]
for every $i \in \NN$. Since $\|a - a_0\| \leq \|a\|/3$, we have $\|a_0\| \ge 2\|a\|/3$,
and so for each $i$ we obtain $\|a_0 \cdot(x_i \times f)\| \ge 2\|a\|/3$. So for $i \in
\NN$, we have
\begin{equation}\label{eq:a-estimate}
\|a \cdot (x_i \times f)\|
    \ge \|a_0 (x_i \times f)\| - \|(a - a_0)(x_i \times f)\|
	\ge  2\|a\|/3 - \|(a - a_0)(x_i \times f)\|.
\end{equation}
Since $\|f\| = 1$, each $\|x_i \times f\| = \|f\circ \alpha^{-1}\|=1$, and since $\|a -
a_0\| \leq \|a\|/3$, we deduce that each $\|(a - a_0)(x_i \times f)\| \le \|a\|/3$. This
and~\eqref{eq:a-estimate} give $\|a\cdot (x_i \times f)\| \geq 2\|a\|/3 - \|a\|/3 =
\|a\|/3$ for all $i$. So, for $a\neq 0$, $\{i\in\NN\colon a\cdot X_i\neq \{0\}\}=\NN$,
which is certainly infinite.
\end{proof}	

We can now prove that $\Tt_X \cong \cs(\ginf)$. First observe that if $\phi\colon
C^*(\ginf)\to B(\Hh)$ is a representation, then we obtain a Toeplitz representation
$(\psi,\pi)$ of $X$ in $\Bb(\Hh)$ by setting
\begin{equation}\label{eq:repdef}
    \psi(x)= \phi(x) \qquad\text{ and }\qquad \pi(a)= \phi(\iotaG(a)).
\end{equation}

\begin{prp} \label{prop:tiscsg}
Let $G$ be an amenable, locally compact, Hausdorff, \'etale groupoid, $\alpha$ an
automorphism of $G$, and \defginf.  Suppose $X$  is the $C^*$-correspondence of
Lemma~\ref{lem:module}. Then there is an isomorphism $\rho : \cs(\ginf) \to  \Tt_X$ such
that $\rho \circ \iotaG = i_{\cs(G)}$.
\end{prp}
\begin{proof}
Choose a faithful representation $\phi$ of $\cs(\ginf)$, and let $(\psi,\pi)$ be the
representation of~\eqref{eq:repdef}. Lemmas \ref{lem:4.3}~and~\ref{lem:4.5} show that the
induced representation $\psi \times \pi$ of $\Tt_X$ is faithful. So it suffices to show
that the image of $\psi \times \pi$ is equal to the image of $\phi$. Clearly we have
$$(\psi \times \pi)(\Tt_X)=C^*\big(i_X(X)\cup i_{\cs(G)}(\cs(G))\big)
=C^*(\phi(X\cup \cs(G)))\subseteq \phi(C^*(\ginf)).$$ To see that $\im \psi \times \pi =
\im \phi$, it suffices to show that the $^*$-subalgebra $\mathcal{A}$ of $C_c(\ginf)$
generated by $X$ is dense in $\cs(\ginf)$. Let
\[
    \mathcal{C}:= \{Z(\mu,\nu) \times U \mid \mu,\nu \in E^* \text{ and } U \text{ is
	an open bisection in }G\},
\]
which is a collection of open bisections of $\ginf$. Note that $\mathcal{A}$ contains all
the functions (with support in $\mathcal{C}$) of the form $1_{Z(\mu,\nu)} \times f$ where
$\supp f \subseteq U$ for some open bisection $U \subseteq G$.

Proposition~3.14 of \cite{Exel} says that each element of $C_c(\ginf)$ can be written as
a finite sum of the form $\sum_n f_n$ where each $f_n$ is supported on some $B_n \in
\mathcal{C}$.  Thus it suffices to show that for $B\in \mathcal{C}$, any element of
$C_c(\ginf)$ supported on $B$ can be approximated by functions in $\mathcal{A}$. We have
$\|f\|_\infty\leq \|f\|$ for all $f\in C_c(\ginf)$ (see, for example,
\cite[Lemma~2.1(1)]{BCS}), and Proposition~3.14 of \cite{Exel} gives the reverse
inequality for $f$ supported on $B\in \mathcal{C}$. So the result follows from the
Stone--Weierstrass theorem.

Finally with $\rho:=(\psi \times \pi)^{-1}\circ \phi$, we use $(\psi \times \pi)\circ i_{\cs(G)}=\pi$
and $\pi=\phi\circ\iotaG$ to deduce that $\rho \circ \iotaG = (\psi \times \pi)^{-1}\circ \pi = i_{\cs(G)}$.
\end{proof}

\begin{cor} \label{cor:samek}
Let $G$ be a second-countable, amenable, locally compact, Hausdorff, \'etale groupoid,
$\alpha$ an automorphism of $G$, and \defginf. Then the $KK$-class of the map $\iotaG :
\cs(G) \to \cs(\ginf)$ of Lemma~\ref{lem:embedding} is a $KK$-equivalence. In particular,
$K_*(\iotaG) : K_*(\cs(G)) \to K_*(\cs(\ginf))$ is an isomorphism.
\end{cor}
\begin{proof}
Let $X$ be the $C^*$-correspondence of Lemma~\ref{lem:module}. Then
Proposition~\ref{prop:tiscsg} gives $\cs(\ginf) \cong \Tt_X$. Since $G$ is second
countable, $\cs(G)$ is separable, and so \cite[Theorem~4.4]{Pimsner} implies that the
$KK$-class of the inclusion $i_{\cs(G)} : \cs(G) \to \Tt_X$ is a $KK$-equivalence. In
particular, $\iotaG$ induces an isomorphism in $K$-theory (see \cite[page~38]{Ror}).
\end{proof}

Although we do not need it in the current paper, we digress now to prove that $\Tt_X$,
and so $\cs(\ginf)$, coincides with $\OO_X$. Let $\varphi_X\colon \cs(G) \to \Ll(X)$
denote the homomorphism implementing the left action on $X$.

\begin{prp}
Let $G$ be a second-countable, amenable, locally compact, Hausdorff, \'etale groupoid,
$\alpha$ an automorphism of $G$, and \defginf. Let $X$ be the $C^*$-correspondence of
Lemma~\ref{lem:module}. Then
\begin{enumerate}
 \item\label{it1:cpa} we have $\varphi_X(\cs(G)) \cap \Kk(X)=\{0\}$; and
\item\label{it2:cpa} we have $C^*(\ginf)\cong \OO_X$.
\end{enumerate}
\end{prp}
\begin{proof}
For~(\ref{it1:cpa}), first identify $X=\bigoplus_i X_i$ using Lemma~\ref{lem:graded}.
Since $\bigoplus_i X_i = \overline{\bigcup_n \bigoplus^n_{i=1} X_i}$, it follows that
$\Kk(X) = \bigcup^\infty_{n=1} \Kk(\bigoplus^n_{i=1} X_i)$. For $T \in
\Kk(\bigoplus^n_{i=1} X_i)$, we have $T (x_j \times f) = 0$ for $j > n$ and $f \in
C_0(\go)$. It then follows from continuity that
\begin{equation}\label{eq:cpt to zero}
\lim_{j \to \infty} \|T (x_j \times f)\| = 0 \qquad\text{ for all $T \in \Kk(X)$ and $f \in C_0(\go)$.}
\end{equation}

Fix $T\in \varphi_X(\cs(G)) \cap \Kk(X)$, say $T=\varphi_X(a)$. By Lemma~\ref{lem:4.5}
there exist $f \in C_0(\go)$ such that $\|a \cdot (x_i \times f)\|\geq\|a\|/3$ for each
$i\in\NN$. Using \eqref{eq:cpt to zero}, we see that
\[
0= \lim_{j \to \infty} \|T (x_j \times f)\| =  \lim_{j \to \infty}\|a \cdot (x_j \times f)\|\geq\|a\|/3,
\]
This forces $a = 0$ and hence $0 = \varphi_X(a) = T$, giving~(\ref{it1:cpa}).
For~(\ref{it2:cpa}), recall that $\cs(\ginf) \cong \Tt_X$ by
Proposition~\ref{prop:tiscsg}, and so the result follows from~(\ref{it1:cpa}) and the
definition of $\OO_X$ (see \cite[Corollary~3.14]{Pimsner}).
\end{proof}

\section{Structural properties of \texorpdfstring{$\ginf$}{the twisted product groupoid}}\label{sec:structure}

Recall that a Kirchberg algebra is a nuclear, separable, simple, purely infinite
$C^*$-algebra.  Our goal in this section is to identify conditions on $G$ under which
$\ginf$ is principal and $\cs(\ginf)$ is a Kirchberg algebra in the UCT class with the
same $K$-theory as $\cs(G)$. We have established that $G$ being amenable forces
$\cs(\ginf)$ to be nuclear and in the UCT class in Lemma~\ref{lem:amenable}. If $G$ is
also second countable, then $\cs(\ginf)$ is separable, and has the same $K$-theory as
$\cs(G)$. Thus, once we find conditions on $G$ under which $\ginf$ is principal and
minimal, all that will remain is to consider pure infiniteness.

\begin{prp}\label{prp:princ}
Let $G$ be a locally compact, Hausdorff, \'etale groupoid, $\alpha$ an automorphism of
$G$, and \defginf. The following are equivalent
\begin{enumerate}
\item\label{it1:propprinc} $\ginf$ is principal,
\item\label{it2:propprinc} $G$ is principal and $\alpha$ induces a free action of
    $\ZZ$ on the orbit-space of $G$ in the sense that
\begin{equation}
\label{wfc}
\text{if $x \in \go$ and $l\in \ZZ$ satisfy } [x] = [\alpha^l(x)] \text{ then } l=0.
\end{equation}
\end{enumerate}
\end{prp}
\begin{proof}
(\ref{it1:propprinc}) $\implies$(\ref{it2:propprinc}): Suppose that $\ginf$ is principal.
Then for each unit $x$ of $H_\infty$, the subgroupoid $\{x\} \times G$ of $\ginf$ is also
principal. Since $(x,g) \mapsto g$ is an isomorphism of $\{x\} \times G$ onto $G$, it
follows that $G$ is principal. To establish~\eqref{wfc}, fix $x \in \go$ and suppose that
$[x] = [\alpha^l(x)]$ for some $l \in \ZZ$. By assumption, there exists $g_0 \in G$ such
that
\[
    s(g_0) = \alpha^l(x) \quad\text{ and }\quad r(g_0)=x.
\]
Recall that $\hio$ is the space of all paths (finite and infinite) of the graph with one
vertex $v$ and infinitely many edges $e_i$. Let $y_0 =  e_0e_0e_0 \dots \in \hio$, and
let $h_0 := (y_0,-l,y_0) \in H_\infty$. Let $\gamma_0 = (h_0,g_0)$. Then
\[
s(\gamma_0) = (y_0, \alpha^{-l}(s(\alpha^{l}(x)))) = (y_0, x) = r(\gamma_0).
\]
Since $\ginf$ is principal we then have $\gamma_0 \in \ginfo$, so $h_0 \in \hio$ forcing
$l=0$.

(\ref{it2:propprinc})$\implies$(\ref{it1:propprinc}): Suppose that~(\ref{it2:propprinc})
holds. Fix $\gamma=(h,g) \in \ginf$ such that $s(\gamma) = r(\gamma)$. Then $s(h)=r(h)$
and $s(\alpha^{c(h)}(g)) = r(g)$. Hence
\[
    [s(g)] = [r(g)] = [s(\alpha^{c(h)}(g))] = [\alpha^{c(h)}(s(g))].
\]
Using~(\ref{wfc}) we see that $c(h)=0$, and hence $h=(r(h),0,s(h)) \in \hio$. Since $G$
is principal, we also have $g \in \go$, and so $\gamma \in \hio \times \go =
(\ginf)^{(0)}$.
\end{proof}

\begin{rmk}
We had to work a little to ensure that $\ginf$ is principal in
Proposition~\ref{prp:princ}; but it is automatically topologically principal whenever $G$
is. To see this, observe that if $G^u_u = \{u\}$ and $(H_\infty)^x_x = \{x\}$, then
$(h,g) \in (\ginf)^{(x,u)}_{(x,u)}$ forces $h \in (H_\infty)^x_x = \{x\}$, and then \[g
\in G^u_{\alpha^{c(h)}(u)} = G^u_u = \{u\}; \text{ so } (h,g) = (x,u) \in
(\ginf)^{(0)}.\] Since the product of dense subsets of $\hio$ and $\go$ is dense in $\hio
\times \go = (\ginf)^{(0)}$, and since $H_\infty$ is topologically principal, it follows
that if $G$ is topologically principal, then so is $\ginf$.
\end{rmk}

We next describe a necessary and sufficient condition for $\ginf$ to be minimal.

\begin{prp}\label{prp:minimal}
Let $G$ be a locally compact, Hausdorff, \'etale groupoid, $\alpha$ an automorphism of
$G$, and \defginf. The groupoid $\ginf$ is minimal if and only if for each $y\in \go$ the
set $\bigcup_{n\leq 0} \alpha^n([y])$ is dense in $\go$. So if $G$ is minimal, so is
$\ginf$.
\end{prp}
\begin{proof}
First suppose $\ginf$ is minimal and choose $x,y\in \go$. Choose an open neighbourhood
$V$ of $x$ in $\go$. Since $[y] = s(G^y)$, it suffices to find $\gamma\in G^y$ and $n \le
0$ such that $\alpha^n(s(\gamma))\in V$. Recall that $H_\infty$ is constructed from the
graph with a single vertex $v$. Since $\ginf$ is minimal there exists
$(h,\gamma)\in\ginf$ such that $(r(h),r(\gamma)) = r(h,\gamma) = (v,y)$ and $(s(h),
\alpha^{c(h)}(s(\gamma))) = s(h,\gamma) \in \hio \times V$. So $\gamma \in G^y$, and
$\alpha^{c(h)}(s(\gamma)) \in V$. Since $r(h) = v \in E^0$, we have $h = (v, -|\delta|,
\delta)$ for some $\delta  \in E^*$. So $n := c(h) = -|\delta| \le 0$ satisfies
$\alpha^n(s(\gamma)) \in V$ as required.

Now suppose that for all $y\in \go$, the set $\bigcup_{n\leq 0} \alpha^n([y])$ is dense
in $\go$. Observe that by applying this with $y = \alpha^p(z)$ we deduce that
\begin{equation}\label{eq:looks easier}\textstyle
    \bigcup_{n \le p} \alpha^n([z])\text{ is dense in $\go$ for every $p \in \ZZ$ and $z \in \go$.}
\end{equation}
Choose $(u,w), (x,y)\in \hio\times \go$ and let $U\times V$ be a basic open set
containing $(u,w)$. We must find $\xi\in \ginf$ such that $r(\xi)=(x,y)$ and $s(\xi)\in
U\times V$. By Lemma~\ref{lem:handy}, there exists $\lambda \in E^*$ such that
$Z(\lambda) \subseteq U$. Using~\eqref{eq:looks easier}, we can choose $n \le -|\lambda|$
and $\gamma\in G$ such that $r(\gamma) = y$ and $\alpha^{n}(s(\gamma))\in V$. Choose
$\lambda' \in E^*$ such that $|\lambda\lambda'| = -n$. Let $h := \big(x, n,
\lambda\lambda' x\big) \in H_\infty$. Then $r(h) = x$ and $s(h) \in Z(\lambda) \subseteq
U$. So $\xi := (h,\gamma)$ satisfies $r(\xi)=(x,y)$ and $s(\xi) = (s(h),
\alpha^{n}(s(\gamma))) \in U\times V$ as required.

The final statement follows immediately: if $G$ is minimal, then $[y] \subseteq
\bigcup_{n \le 0} \alpha^n([y])$ is dense in $\go$ for every $y$.
\end{proof}

The idea of our construction is that the $C^*$-algebra $\cs(\ginf)$ is something like a
twisted tensor product of $C^*(G)$ with $\OO_\infty$, and so we can expect it frequently
to be purely infinite even when $C^*(G)$ is not. While there is not as yet a completely
satisfactory characterisation of the groupoids $G$ whose $C^*$-algebras are purely
infinite (but see \cite{BCS}), there is a very useful sufficient condition due to
Anantharaman-Delaroche \cite[Proposition~2.4]{AD}. Our next result provides conditions on
$G$ and $\alpha$ that ensure $\ginf$ satisfies this condition. Our proof requires the
additional assumption that $G$ is ample (this is our main source of examples in any
case); but we expect that a similar result should hold more generally.

Recall that an ample groupoid $L$ is \emph{locally contracting} if for every nonempty open
 set $W\subseteq L^{(0)}$, there exists a compact open bisection $B$ such that
\[
    r(B) \subsetneq s(B) \subseteq W.
\]
The groupoid $H_\infty$ is certainly locally contracting: by Lemma~\ref{lem:handy}, each
nonempty open set in $W \subseteq \hio$ contains a set of the form $Z(\lambda)$, and then
$B = Z(\lambda e_1, \lambda)$ satisfies $r(B) = Z(\lambda e_1) \subsetneq Z(\lambda) =
s(B) \subseteq W$.

\begin{lem}\label{lem:lc}
Let $G$ be a locally compact, Hausdorff, \'etale groupoid, $\alpha$ an automorphism of
$G$, and \defginf.  Suppose that $G$ is ample and $\Bb$ is a basis for the topology on
$G^{(0)}$ consisting of compact open bisections such that
\begin{equation}
\label{cond:lc}\text{for every $V\in\Bb$ there exists $l \ge 1$ such that $\alpha^{-l}(V)\subseteq V$}.
\end{equation}
Then $\ginf$ is locally contracting.
\end{lem}
\begin{proof}
Let $W$ be a nonempty open set in $(\ginf)^{(0)}=\hio\times \go$. Then there exist nonempty clopen subsets
$V_H\subseteq \hio$ and $V_G\subseteq \go$ such that $V_G\in \Bb$ and $V_H\times
V_G\subseteq W$. By Lemma~\ref{lem:handy}, there exists $\lambda \in E^*$ such that
$Z(\lambda)\subseteq V_H$. Since each $Z(\lambda e_i) \subseteq Z(\lambda)$, we can
assume that $|\lambda| \ge 1$. Define $N:=|\lambda|$.

Fix $l \ge 1$ such that $\alpha^{-l}(V_G)\subseteq V_G$.  Then
$\alpha^{-Nl}(V_G)\subseteq V_G$. Let $\lambda^{l}$ denote the finite path
$\lambda\lambda \cdots\lambda$ in which $\lambda$ is repeated $l$ times. Let
$U:=Z(\lambda^{2l}, \lambda^l)$. Then
\[
    r(U)=Z(\lambda^{2l})\text{ and }s(U) = Z(\lambda^l),
\]
and we also have $U \subseteq c^{-1}(Nl)$.

Let $B := U\times \alpha^{-Nl}(V_G) \subseteq \ginf$. Then $B$ is a compact open
bisection, and
\[
r(B)
    =r(U)\times \alpha^{-Nl}(V_G)
    \subsetneq Z(\lambda^l)\times V_G
    =Z(\lambda^l)\times \alpha^{Nl}(\alpha^{-Nl}(V_G))
    =s(B)
    \subseteq W.
\]
Hence $\ginf$ is locally contracting.
\end{proof}

Putting all this together, we obtain our main result.


\begin{thm}\label{thm:main}
Let $G$ be a second-countable, amenable, locally compact, Hausdorff, \'etale groupoid,
$\alpha$ an automorphism of $G$ satisfying conditions (\ref{wfc})~and~(\ref{cond:lc}), and \defginf.
Then $\ginf$ is also a second-countable, amenable, locally compact, Hausdorff, \'etale groupoid. Furthermore
\begin{enumerate}
\item\label{it:same Kth} the inclusion $\iotaG : \cs(G) \to \cs(\ginf)$ of
    Lemma~\ref{lem:embedding} induces a $KK$-equivalence;
\item\label{it:principal} if $G$ principal, then  so is $\ginf$;
\item\label{it:ample} if $G$ ample, then $\ginf$ is ample and locally contracting;
    and
\item\label{it:minimal} if $G$ minimal, so is $\ginf$.
\end{enumerate}
In particular, if $G$ is principal, ample and minimal, then $C^*(\ginf)$ is
a Kirchberg algebra in the UCT class with the same $K$-theory as $\cs(G)$.
\end{thm}

\begin{proof}
The first statement and~(\ref{it:ample}) follow from Remark~\ref{rem.2.1} and
Lemma~\ref{lem:etale}, Lemma~\ref{lem:amenable} and Lemma~\ref{lem:lc}.
Item~(\ref{it:same Kth}) follows from Corollary~\ref{cor:samek},
item~(\ref{it:principal}) follows from Lemma~\ref{prp:princ}, and item~(\ref{it:minimal})
follows from the final statement of Proposition~\ref{prp:minimal}.

For the final statement observe that $C^*(\ginf)$ is nuclear and in the UCT class by
Lemma~\ref{lem:amenable}. It is separable because $\ginf$ is second countable by Lemma
\ref{lem:etale}. Items (\ref{it:principal})~and~(\ref{it:minimal}) combine with
\cite[Corollary~4.6]{Renault:JOT91} (see also \cite[Theorem~5.1]{BCFS}) to show that
$\cs(\ginf)$ is simple, and Lemma~\ref{lem:lc} combines with \cite[Proposition~2.4]{AD}
to show that it is purely infinite. Item~(\ref{it:same Kth}) (or
Corollary~\ref{cor:samek}) says that $C^*(\ginf)$ has the same $K$-theory as $\cs(G)$,
completing the proof.
\end{proof}

It will be useful to be able to stabilise the $C^*$-algebras of the groupoids $G$
appearing in Theorem~\ref{thm:main}. The following result allows us to do so.
We define $K$ to be the complete equivalence relation on $\ZZ$ regarded as a discrete
groupoid. So $\cs(K)$ is canonically isomorphic to $\Kk(\ell^2(\ZZ))$ (see for example Theorem~3.1 of \cite{MRW87}).

\begin{lem}\label{lem:stabilise}
Suppose that $G$ is a second-countable, amenable, locally compact, Hausdorff, \'etale
groupoid, and let $\alpha$ be an automorphism of $G$ that satisfies conditions
(\ref{wfc})~and~(\ref{cond:lc}). Then the cartesian product $G \times K$ and the
automorphism $\alpha \times \id : G \times K \to G \times K$ given by
\[
(\alpha \times \id)(g, (m,n)) = (\alpha(g), (m,n))
\]
also have all these properties. Moreover,
\[
    \cs(G \times K) \cong C^*(G) \otimes \Kk(\ell^2(\ZZ)),
\]
and $G \times K$ is principal if $G$ is principal, minimal if $G$ is minimal, and ample
if $G$ is ample.
\end{lem}
\begin{proof}
The cartesian product of second-countable locally compact Hausdorff spaces is itself
second-countable, locally compact and Hausdorff. Since $r : G \to \go$ and $r' : K \to
K^{(0)}$ are local homeomorphisms the product map $r \times r' : G \times K \to (G \times
K)^{(0)}$ is too, and so $G \times K$ is \'etale. Since $G$ is amenable,
$C^*(G)=C_r^*(G)$ is nuclear. By definition of $K$, the associated $C^*$-algebra
$C^*(K)=C_r^*(K)$ is nuclear and simple. Using \cite[Proposition~10.1.5]{BO}, we see that
$C_r^*(G \times K)\cong C_r^*(G)\otimes C_r^*(K)$ is nuclear, and then $G \times K$ is
amenable by \cite[Corollary~6.2.14(ii)]{A-DR}. Since $\cs(K)$ is canonically isomorphic
to $\Kk(\ell^2(\ZZ))$, the $\cs$-algebra $C^*(G \times K)$ is canonically isomorphic to
$C^*(G) \otimes C^*(K) = C^*(G)\otimes \Kk(\ell^2(\ZZ))$. For the final statement, note
that since $K$ is principal, minimal and ample, taking cartesian products with $K$
preserves these properties.
\end{proof}

\section{Examples from Bratteli diagrams}\label{sec:1-graph examples}
The hypotheses of Theorem~\ref{thm:main} may appear restrictive, but we show in this and
the subsequent section that there are many examples. By the Kirchberg--Phillips
classification theorem, for each simple dimension group $D \not= \ZZ$ and each $d \in
D^+$ there is a unique Kirchberg algebra $A_D$ with  $(K_0(A_D), [1_{A_D}], K_1({A_D}))
\cong (D, d, \{0\})$. In this section we show that $A_D$ can be realised as the
$C^*$-algebra of an amenable, principal, ample groupoid with compact unit space.

\begin{thm}\label{thm:AFeg}
Let $D$ be a simple dimension group other than $\ZZ$. Let $A$ be a UCT Kirchberg algebra
with trivial $K_1$-group and suppose that there is an isomorphism $\pi : K_0(A) \to D$.
If $A$ is unital, suppose further that $\pi([1_A]) \in D^+$. Then there is an amenable,
principal, ample groupoid $\Gg$ such that $A \cong \cs(\Gg)$. If $A$ is unital, then
$\Gg^{(0)}$ is compact.
\end{thm}

The remainder of this section will be spent proving Theorem~\ref{thm:AFeg}. Our
convention is that a \emph{Bratteli diagram} is a row-finite directed graph whose vertex
set $E^0$ is partitioned into finite sets $V_n$, $n = 0, 1, 2, \dots$ and whose edge set
is partitioned as $E^1 = \bigcup^\infty_{n=0} V_n E^1 V_{n+1}$. We will further assume
that $vE^1 \not= \emptyset$ for all $v \in E^0$ and that $E^1v \not= \emptyset$ for all
$v \in E^0 \setminus V_0$.

\begin{lem}\label{lem.6.2}
Let $D$ be a simple dimension group other than $\ZZ$. Then there exists a Bratteli
diagram $E$ such that $K_0(C^*(E))\cong D$, $K_1(C^*(E))\cong \{0\}$ and $k_{vw}:=|vEw| >
n$ for each $n\in \NN$, $v\in V_n$ and $w\in V_{n+1}$. Let $G$ denote the graph
groupoid~\eqref{eqn.2.2} of $E$. Then $G$ is a second-countable, amenable, locally
compact, Hausdorff, ample groupoid that is minimal and principal.
\end{lem}
\begin{proof}
Let $B$ be an AF algebra whose $K_0$-group is isomorphic to $D$. Then $K_1(B) = \{0\}$
because $B$ is AF. By \cite[Corollary IV.5.2]{Dav}, $B$ is simple. Let $E$ be a Bratteli
diagram for $B$. By \cite[Theorem~1]{Tyler} (see also \cite[Theorem~1]{Drinen}), $B$ is
Morita equivalent to $\cs(E)$. By telescoping $E$, we can assume that for every $n \in
\NN$, every $v\in V_n$ and every $w\in V_{n+1}$, we have $k_{vw}=|vEw| > n$: this
follows, for example, from Lemmas A4.3~and~A4.4 of \cite{Effros}---see the proof of
Theorem~6.2(2) on page~163 of \cite{PRRS}.

Since $G$ is the path groupoid \eqref{eqn.2.2} associated to $E$, it is a
second-countable, amenable, locally compact, Hausdorff, ample groupoid, and we have
$\cs(E) \cong \cs(G)$ (see Remark~\ref{rem.2.1}). Since simplicity is preserved by Morita
equivalence and $B$ is simple, $\cs(G)$ is simple and hence $G$ is minimal (see, for
example, \cite[Theorem~5.1]{BCFS}). Since $E$ has no cycles, $G$ is also principal by
\cite[Proposition~8.1]{aHH}.
\end{proof}

With $E$, $k_{vw}$ and $G$ as in Lemma \ref{lem.6.2}, we label the edges in $vEw$ as
$\{(vw)_{1}, \dots, (vw)_{k_{vw}}\}$. To define an automorphism $\alpha$ of $G$, first
consider the graph automorphism, also called $\alpha$, of $E$ such that
\[
\alpha((vw)_{i}) = (vw)_{((i+1)\ \mathrm{mod}\ k_{vw})}\quad\text{ for all $v\in V_n$, $w \in
    V_{n+1}$ and $i < k_{vw}$.}
\]
Then $\alpha$ restricts to a permutation of the set $vE^1w$ for $v\in V_n$ and $w\in
V_{n+1}$, and $\alpha$ pointwise fixes $E^0$.

\begin{lem}\label{lem.6.3}
With $G$ as in Lemma \ref{lem.6.2}, and $\alpha$ as above, there is an automorphism
$\alpha$ of $G$ satisfying \eqref{wfc}  and \eqref{cond:lc}  such that $\alpha(x,m,y) = (\alpha(x),m,\alpha(y))$ for $(x,m,y)\in G$.
\end{lem}
\begin{proof}
Since $\alpha$ is an automorphism of $E$, it determines a bijection of $P = P_E$ by
$\alpha(x)_i = \alpha(x_i)$ for $i \le |x|$. This map clearly commutes with $\sigma^n$
for all $n$, and so determines an algebraic automorphism $\alpha$ of $G$ as described.
Since this automorphism carries each basic open set $Z((\mu,\nu)\setminus F)$ bijectively
onto $Z\big((\alpha(\mu),\alpha(\nu))\setminus\alpha(F)\big)$, it is a homeomorphism, and
hence an automorphism of the topological groupoid $G$. In particular, we have
$\alpha^l(Z(\mu)) = Z(\alpha^l(\mu))$. Since the orbit of any $\mu \in E^*$ under
$\alpha$ is contained in the finite set $r(\mu) E^{|\mu|} s(\mu)$, it is finite, and so
there exists $l > 0$ such that $\alpha^{-l}(\mu) = \mu$, giving $\alpha^{-l}(Z(\mu)) =
Z(\mu)$. Since $E$ is row-finite, $P_E = E^\infty$, so $G$ and $\alpha$
satisfy~\eqref{cond:lc} with $\Bb = \{Z(\mu) : \mu \in E^*\}$.

It remains to show $\alpha$ satisfies condition~(\ref{wfc}). Fix $x \in \go = E^{\infty}$ and suppose $[x]=[\alpha^l(x)]$ for some $l \in \ZZ$. Then
there exist $m, n$ such that
\[
    \sigma^m(x) = \sigma^n(\alpha^l(x)).
\]
Thus, $r(\sigma^m(x))=r(\sigma^n(\alpha^l(x)))$. We have $r(x) \in V_j$ for some $j$, and
since $\alpha$ fixes vertices of the graph $E$, we then have $r(\alpha^l(x)) = r(x) \in V_j$ as
well. So $V_{j+m} \owns r(\sigma^m(x)) = r(\sigma^n(\alpha^l(x))) \in V_{j+n}$. Since the
$V_i$ are mutually disjoint, we obtain $m=n$. By replacing $x$ with $\sigma^m(x)$, we may
assume that $x=\alpha^l(x)$. For each $p \in \NN$, we have $x(p,p+1) =
\alpha^l((x)(p,p+1))$. For each $p$, let $v_p = r(x(p, p+1))$ and $w_p= s(x(p,p+1))$.
Then $l=0 \mod k_{v_pw_p}$ for all $p$. Lemma~\ref{lem.6.2} shows that $k_{v_pw_p} \to
\infty$, and it follows that $l = 0$.
\end{proof}

\begin{proof}[Proof of Theorem~\ref{thm:AFeg}]
Let $E$ be the Bratteli diagram of Lemma~\ref{lem.6.2} and let $\alpha$ be the
automorphism of the path groupoid $G$ of $E$ described in Lemma~\ref{lem.6.3}. By
Lemmas~\ref{lem.6.2}~and~\ref{lem.6.3}, we have $K_*(A)\cong K_*(C^*(G))$,
the groupoid $G$ is second-countable, amenable, locally compact, Hausdorff, ample,
minimal and principal, and $\alpha$ satisfies conditions \eqref{wfc}~and~(\ref{cond:lc}).

Now Lemma~\ref{lem:stabilise} implies that $\widetilde{G} = G \times K$ and
$\widetilde{\alpha} = \alpha \times \id$ have all the properties of $G$ and $\alpha$
discussed in the preceding paragraph. It also implies that $\cs(\widetilde{G})$ is
stable. Applying Theorem~\ref{thm:main}, we obtain a second-countable, amenable, locally
compact, Hausdorff, ample groupoid $\tginf$ that is minimal and principal and
 an inclusion $\iota_{\widetilde{G}} : \cs(\widetilde{G}) \to \cs(\tginf)$ such that
$K_*(\iota_{\widetilde{G}})$ is an isomorphism and $\cs(\tginf)$ is a Kirchberg algebra
in the UCT class.

First suppose that $A$ is nonunital. Since $\Gg := \tginf$ has noncompact unit space,
$\cs(\tginf)$ is also nonunital. So $A$ and $\cs(\Gg)$ are nonunital UCT Kirchberg
algebras with the same $K$-theory because
\[
    K_*(A)\cong K_*(C^*(E))\cong K_*(C^*(G))\cong K_*(C^*(\widetilde{G}))\cong K_*(C^*(\tginf)) = K_*(C^*(\Gg)).
\]
Hence $A\cong\cs(\Gg)$ by the Kirchberg--Phillips classification theorem \cite{Kirchberg,
Phillips}.

Now suppose that $A$ is unital. There is an order isomorphism $\phi : K_0(C^*(E)) \cong
\varinjlim (\ZZ V_n, B_n)$ where $B_n \in M_{V_{n+1}, V_n}(\NN)$ is given by $B_n(v,w) =
|vE^1w|$ (see, for example, \cite{Drinen}). So $D = \varinjlim(\ZZ V_n, B_n)$, and $D^+$
is the union of the images of the subsemigroups $\NN V_n \subseteq \ZZ V_n$. Writing
$B_{n,\infty} : \ZZ V_n \to \varinjlim(\ZZ V_n, B_n)$ for the canonical inclusion, we
have $\phi([p_v]) = B_{n,\infty}(\delta_v)$ for $v \in V_n$. Since the isomorphism
$C^*(G) \cong C^*(E)$ carries $1_{Z(v)}$ to $p_v$ \cite{KPRR}, we obtain an isomorphism
$\tilde\phi : K_0(C^*(G)) \cong D$ such that $\tilde\phi(1_{Z(v)}) =
B_{n,\infty}(\delta_v)$ for $v \in V_n$.

By assumption, $\pi([1_A]) \in D^+$, so there exist $n \ge 1$ and $a \in \NN V_n$ such
that $\pi([1_A]) = B_{n,\infty}(a)$. Define $V \subseteq \widetilde{G}^{(0)} = G^{(0)}
\times \ZZ$ by $V := \bigcup_{v \in V_n} \bigcup_{1 \le j \le a(v)} Z(v) \times \{j\}$;
observe that $V$ is clopen. The canonical isomorphism $K_0(\cs(\widetilde{G})) \cong
K_0(\cs(G))$ induced by the isomorphism $\cs(\widetilde{G}) \cong \cs(G) \otimes
\Kk(\ell^2(\ZZ))$ carries $[1_V]$ to $\sum_{v \in V_n} a(v)[1_{Z(v)}]$. So composing this
isomorphism with $\tilde\phi$ gives an isomorphism $K_0(C^*(\widetilde{G})) \cong D$ that
takes $[1_V]$ to $\pi([1_A])$.

Since $K_*(\iota_{\widetilde{G}}) : K_*(\cs(\widetilde{G})) \to K_*(\cs(\tginf))$ is an
isomorphism, it follows that
\[
    (K_0(A), 1_A) \cong (K_0(\cs(\tginf)), [1_W]),
\]
where $1_W=\iota_{\widetilde{G}}(1_V)$. Let $\Gg$ be the restriction $\tginf|_{W} = \{g
\in \tginf : r(g), s(g) \in W\}$ of $\tginf$ to the compact open subset $W$ of its unit
space. Then $\cs(\Gg) \cong 1_{W} \cs(\tginf) 1_{W}$ is a corner of $\cs(\tginf)$, which
is full since $\cs(\tginf)$ is simple. We therefore have
\[
(K_0(\cs(\Gg)), 1_{\cs(\Gg)}) \cong (K_0(A), 1_A).
\]
Since $\tginf$ is amenable, principal and ample, so is $\Gg$. Since simplicity,
nuclearity, separability, pure infiniteness and membership of the UCT class pass to full
corners, $\cs(\Gg)$ is a UCT Kirchberg algebra. So the Kirchberg--Phillips theorem
\cite{Kirchberg, Phillips} gives $\cs(\Gg) \cong A$ as required.
\end{proof}

\section{Examples from rank-2 Bratteli diagrams}\label{sec:2-graph examples}

In this section, we show that Theorem~\ref{thm:main} can also be applied to groupoids
associated to the rank-2 Bratteli diagrams of \cite{PRRS}.
Recall that a matrix $A$ with nonnegative entries is \emph{proper} if each row and each
column of $A$ has at least one nonzero entry.

\begin{thm}\label{thm8.1}
Let $\{c_n: n\in \NN\}$ be positive integers. For each $n$, let $A_n,B_n\in
M_{c_{n+1},c_n}(\NN)$ be proper matrices, and let $T_n\in M_{c_n}(\NN)$ be a proper
diagonal matrix. Suppose that $A_nT_n=T_{n+1}B_n$ for all $n$ and $\varinjlim (\ZZ^{c_n},
A_n)$ is a simple dimension group not isomorphic to $\ZZ$. Let $A$ be a Kirchberg algebra
such that $K_0(A) \cong \varinjlim (\ZZ^{c_n}, A_n)$ and $K_1(A) \cong \varinjlim
(\ZZ^{c_n}, B_n)$. Further suppose, if $A$ is unital, that the isomorphism $K_0(A) \cong
\varinjlim (\ZZ^{c_n}, A_n)$ carries $[1_A]$ to an element of the canonical positive cone
of $\varinjlim (\ZZ^{c_n}, A_n)$. Then there exists an amenable, principal, ample
groupoid $\Gg$ such that $C^*(\Gg) \cong A$. If $A$ is unital, then $\go$ is compact.
\end{thm}

The main ingredients in the proof of Theorem~\ref{thm8.1} are Lemmas
\ref{lem8.2}--\ref{lem8.4}. We briefly recall the notions involved, starting with rank-2
Bratteli diagrams.
For a row-finite 2-graph $\Lambda$ we will identify the \emph{blue graph} $f_1^*\Lambda
:= \{(\lambda,n): n\in \NN, \lambda\in \Lambda^{n e_1}\}$ discussed in \cite{PRRS} with
the subgraph $\Lambda^{\NN e_1}\subseteq \Lambda$, and similarly we identify the
\emph{red graph} $f_2^*\Lambda$ with $\Lambda^{\NN e_2}\subseteq \Lambda$. Following
\cite{PRRS}, a \emph{rank-2 Bratteli diagram} (of infinite depth) is a row-finite 2-graph
$\Lambda$ such that $\Lambda^0$ is a disjoint union of nonempty finite sets or
\emph{levels} $(V_n)_{n=0}^\infty$ which satisfy:
\begin{itemize}
\item[(1)] for every blue edge $e\in \Lambda^{e_1}$, there exists $n$ such that $e
    \in V_n E^1 V_{n+1}$;
\item[(2)] the blue graph has no sources and all sinks belong to $V_0$; and
\item[(3)] every vertex of $\Lambda$ lies on an \emph{isolated cycle}\footnote {A
    \emph{cycle} in a $k$-graph $\Lambda$ is an element $\lambda\in\Lambda$ with the
    property that $d(\lambda)\neq 0$, $r(\lambda)=s(\lambda)$ and if $\lambda =
    \lambda'\lambda''$ with $0 < d(\lambda') < d(\lambda)$, then $s(\lambda')\neq
    s(\lambda)$. The cycle $\lambda$ is \emph{isolated} if, whenever $\lambda =
    \lambda'\lambda''$ with $d(\lambda') \not= 0$, the sets
    \[
    r(\lambda)\Lambda^{d(\lambda')}\setminus \{\lambda'\} \qquad\textrm{and}\qquad
    \Lambda^{d(\lambda'')} s(\lambda)\setminus \{\lambda''\}
    \]
    are both empty. See \cite[p.~140]{PRRS}.} of the red graph, and $\Lambda^{e_2} =
    \bigcup_n V_n \Lambda^{e_2}V_n$.
\end{itemize}
The $\cs$-algebra $C^*(\Lambda)$ of a rank-2 Bratteli diagram $\Lambda$ is an A$\TT$
algebra. Its $K$-groups can be computed as follows: Recall that every $V_n$ is a disjoint
union of sets $\bigcup_{i=1}^{c_n} V_{n,i}$ where each $V_{n,i}$ consists of the vertices
on an isolated red cycle. For each $0\leq n$, $1\leq j\leq c_n$, and $1\leq i\leq
c_{n+1}$ define
\[
A_n(i,j):=|\{v\Lambda^{e_1}V_{n+1,i}\}|\quad\text{ and }\quad B_n(i,j):=|\{V_{n,j}\Lambda^{e_1}w\}|,
\]
for any choice of $v\in V_{n,j}$, $w\in V_{n+1,i}$ (the result is independent of these
choices). By \cite[Theorem 4.3]{PRRS}, there is an isomorphism $K_0(C^*(\Lambda))\cong
\lim (\ZZ^{c_n}, A_n)$ that carries $s_v$ to $\delta_j \in \ZZ^{c_n}$ if $v \in V_{n,j}
\subseteq V_n$\label{pg:isodesc}, and $K_1(C^*(\Lambda))\cong \lim (\ZZ^{c_n}, B_n)$.

For each $n$, define a diagonal matrix $T_n\in M_{c_n}(\NN)$ by $T_n(j,j):=|V_{n,j}|$.
Then $A_n T_n = T_{n+1} B_n$ for all $n$ (see \cite[Lemma~4.2]{PRRS}).

In a rank-2 Bratteli diagram $\Lambda$ the factorisation property induces a permutation
$\Ff$ of the edges of the blue graph: For each $e\in \Lambda^{e_1}$, let $f$ be the
unique element of $\Lambda^{e_2} r(e)$, and define $\Ff(e)$ to be the unique element of
$\Lambda^{e_1}$ such that $fe=\Ff(e)f'$ for some red edge $f'$. The \emph{order} $o(e)$
of $e$ is then defined to be the the smallest $k>0$ such that $\Ff^k(e)=e$. Observe that
if $e \in V_n$, then $o(e)\leq |V_n\Lambda^{e_1}|$ is finite since $\Lambda$ is
row-finite and $V_n$ is finite. For each $n \in \NN$, define
\[
    O_n:=\operatorname{lcm}\{o(e)  \mid e \in V_n\Lambda^{e_1}\},
\]
and $m_n$ (inductively) by $m_0:=0$ and $m_{n+1}:= m_n + nO_n$ for $n \ge 0$ (see
Figure~\ref{fig1} for an illustrative example).

\begin{lem}\label{lem8.2}
Let $\{c_n: n\in \NN\}$ be positive integers. For each $n$, let $A_n,B_n\in
M_{c_{n+1},c_n}(\NN)$ be proper matrices, and let $T_n\in M_{c_n}(\NN)$ be a proper
diagonal matrix. Suppose that $A_nT_n=T_{n+1}B_n$ for all $n$ and that $\lim (\ZZ^{c_n},
A_n)$ is a simple dimension group not isomorphic to $\ZZ$. Then there exists a rank-2
Bratteli diagram $\Lambda$ with $\Lambda^0=\bigcup_{n=0}^\infty V_{n}$ such that
$K_0(C^*(\Lambda))\cong \lim (\ZZ^{c_n}, A_n)$, $K_1(C^*(\Lambda))\cong \lim (\ZZ^{c_n},
B_n)$ and	
$$o(e)>n\cdot m_n,$$
for every $n\in\NN$ and $e\in V_n\Lambda^{e_1}$.
\end{lem}
\begin{proof}
We follow the proof of \cite[Theorem 6.2(2)]{PRRS} modulo a small adjustment of the
construction of the subsequence $(l(n))_{n=1}^\infty$ which we define as follows:

For $n\geq m$ define $A_{n,m}:=A_{n-1}A_{n-2}\dots A_m$. As in \cite{PRRS} we can find a
subsequence $(l'(i))_{i=1}^\infty$ of $\NN$ with the property that every entry of
$A_{l'(n+1),l'(n)}$ is at least $n$, and hence
\begin{eqnarray}\label{eqn8}
\lim_{k\to \infty} \min_{p,q}\{A_{l'(i+k),l'(i)}(p,q)\} = \infty\qquad\textrm{ for all }i.
\end{eqnarray}
Let $l(0):=l'(1)$. Using induction we construct sequences $M_n$ and $l(n+1)$, for all
$n\in\NN$. Step $0$: Put $M_0:=0$ and $l(0+1):=l'(2)$. Step $1$: Put $M_1:=0$ and
$l(1+1):=l'(3)$. For the induction step $n+1$, put $M_{n+1}:=M_n+n \cdot
\prod_{i,j}A_{l(n+1),l(n)}(i,j)T_{l(n)}(j,j)$ and use (\ref{eqn8}) to find
$l(n+2)>l(n+1)$ such that every entry of $A_{l(n+2),l(n+1)}$ is strictly grater than
$(n+1)M_{n+1}$.

Proceeding as in the proof of \cite[Theorem 6.2(2)]{PRRS} we let $\Lambda$ be the rank-2
Bratteli diagram obtained by applying \cite[Proposition 6.4]{PRRS} to the data
\[
    c'_n:=c_{l(n)}, \quad A'_n:=A_{l(n+1),l(n)},\quad B'_n:=B_{l(n+1),l(n)} \quad \text{ and } \quad T'_n:=T_{l(n)}
\]
for $n\in \NN$. As in  \cite{PRRS} it follows that $K_0(C^*(\Lambda))\cong \lim
(\ZZ^{c_n}, A_n)$ and $K_1(C^*(\Lambda))\cong \lim (\ZZ^{c_n}, B_n)$.

We now prove that $o(e)>n\cdot M_n \geq n\cdot m_n$ for every $n\in\NN$ and $e\in
V_n\Lambda^{e_1}$ using induction. The statement is trivial when $n=0$ and 1 because each
$o(e) \ge 1$ and $M_1=m_1=0$. For the induction step, we first consider $e\in
V_{n}\Lambda^{e_1}$, say $e \in V_{{n},j}\Lambda V_{{n+1},i}$. By \cite[Proposition
6.4]{PRRS},
\begin{eqnarray}\label{eqn9}
o(e)=A'_{n}(i,j)|V_{{n},j}|=A'_{n}(i,j)T'_{n}(j,j)=A_{l(n+1),l(n)}(i,j)T_{l(n)}(j,j).
\end{eqnarray}
Hence
\[
    m_{n+1}=m_n+n\cdot \operatorname{lcm}\{o(e)  \mid e \in V_n\Lambda^{e_1}\} \leq M_n+ n\cdot \prod_{i,j}A_{l(n+1),l(n)}(i,j)T_{l(n)}(j,j)= M_{n+1},
\]
where we have used that $m_n\leq M_n$.

Secondly we consider any $e\in V_{n+1}\Lambda^{e_1}$. Since $T_{l(n)}(j,j)$ is positive,
we know from~\eqref{eqn9} that $o(e)\geq A_{l(n+2),l(n+1)}(i,j)$ for some (in fact for
any) $i,j$. By construction every entry of $A_{l(n+2),l(n+1)}$ is strictly grater than
$(n+1)M_{n+1}$, so
\[o(e)>(n+1)M_{n+1}\geq (n+1)\cdot m_{n+1}.\qedhere\]
\end{proof}

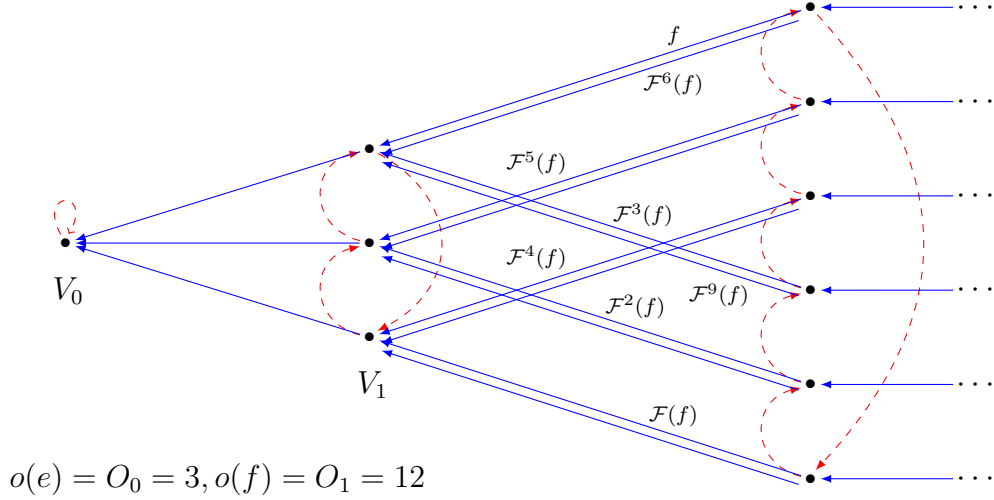
\begin{figure}[ht]
\begin{center}
\begin{tikzpicture}[xscale=2, yscale=1.25]
    \node[inner sep=0.5pt, circle] (05) at (0,7.5) {\tiny$\bullet$};
    \node[inner sep=0.5pt, circle] (24) at (2,6.5) {\tiny$\bullet$};
    \node[inner sep=0.5pt, circle] (24b) at (2,6.4) {$ \ \ $};
    \node[inner sep=0.5pt, circle] (25) at (2,7.5) {\tiny$\bullet$};
    \node[inner sep=0.5pt, circle] (25b) at (2,7.4) {$ \ \ $};
    \node[inner sep=0.5pt, circle] (26) at (2,8.5) {\tiny$\bullet$};
    \node[inner sep=0.5pt, circle] (26b) at (2,8.4) {$ \ \ $};
    \node[inner sep=0.5pt, circle] (f) at (4,9.7) {\tiny{$f$}};	
    \node[inner sep=0.5pt, circle] (f) at (4,5.67) {\tiny{$\Ff(f)$}};
    \node[inner sep=0.5pt, circle] (f) at (3.75,6.83) {\tiny{$\Ff^2(f)$}};
    \node[inner sep=0.5pt, circle] (f) at (3.8,7.8) {\tiny{$\Ff^3(f)$}};	
    \node[inner sep=0.5pt, circle] (f) at (3.1,7.35) {\tiny{$\Ff^4(f)$}};	
    \node[inner sep=0.5pt, circle] (f) at (3.1,8.35) {\tiny{$\Ff^5(f)$}};	
    \node[inner sep=0.5pt, circle] (f) at (4,9.19) {\tiny{$\Ff^6(f)$}};
    \node[inner sep=0.5pt, circle] (f) at (4.3,6.95) {\tiny{$\Ff^9(f)$}};
    \node[inner sep=0.5pt, circle] (f) at (0.02,7.0) {$V_0$};
    \node[inner sep=0.5pt, circle] (f) at (2.02,6.0) {$V_1$};
	\node (f) at (1,5) {$o(e)=O_0=3, o(f)=O_1=12$};
    \node[inner sep=0.5pt, circle] (45) at (4.9,5) {\tiny$\bullet$};
    \node[inner sep=0.5pt, circle] (45b) at (4.9,4.9) {$ \ \ $};
    \node[inner sep=0.5pt, circle] (55) at (6,5) {$\cdots$};
    \node[inner sep=0.5pt, circle] (46) at (4.9,6) {\tiny$\bullet$};
	\node[inner sep=0.5pt, circle] (46b) at (4.9,5.9) {$ \ \ $};
    \node[inner sep=0.5pt, circle] (56) at (6,6) {$\cdots$};
    \node[inner sep=0.5pt, circle] (47) at (4.9,7) {\tiny$\bullet$};
    \node[inner sep=0.5pt, circle] (47b) at (4.9,6.9) {$ \ \ $};
    \node[inner sep=0.5pt, circle] (57) at (6,7) {$\cdots$};
    \node[inner sep=0.5pt, circle] (48) at (4.9,8) {\tiny$\bullet$};
    \node[inner sep=0.5pt, circle] (48b) at (4.9,7.9) {$ \ \ $};
    \node[inner sep=0.5pt, circle] (58) at (6,8) {$\cdots$};
    \node[inner sep=0.5pt, circle] (49) at (4.9,9) {\tiny$\bullet$};
    \node[inner sep=0.5pt, circle] (49b) at (4.9,8.9) {$ \ \ $};
    \node[inner sep=0.5pt, circle] (59) at (6,9) {$\cdots$};
    \node[inner sep=0.5pt, circle] (410) at (4.9,10) {\tiny$\bullet$};
	\node[inner sep=0.5pt, circle] (410b) at (4.9,9.9) {$ \ \ $};
    \node[inner sep=0.5pt, circle] (510) at (6,10) {$\cdots$};
%
    \draw[-latex, red, dashed] (410) edge[out=300,in=60] (45); 
    \draw[-latex, red, dashed] (45) edge[out=160,in=200] (46);
    \draw[-latex, red, dashed] (46) edge[out=160,in=200] (47);
    \draw[-latex, red, dashed] (47) edge[out=160,in=200] (48);
    \draw[-latex, red, dashed] (48) edge[out=160,in=200] (49);
    \draw[-latex, red, dashed] (49) edge[out=160,in=200] (410);
%
    \draw[-latex, red, dashed] (24) edge[out=160,in=200] (25);
    \draw[-latex, red, dashed] (25) edge[out=160,in=200] (26);
    \draw[-latex, red, dashed] (26) edge[out=320,in=40] (24);
	\path[->,every loop/.style={looseness=10}] (05)
	         edge  [in=70,out=110,loop, red, dashed] ();
%
    \draw[-latex, blue] (410)--(26); 
    \draw[-latex, blue] (48)--(24);
    \draw[-latex, blue] (47)--(26);
    \draw[-latex, blue] (46)--(25);
    \draw[-latex, blue] (45)--(24);
    \draw[-latex, blue] (49)--(25);
    \draw[-latex, blue] (410b)--(26b); 
    \draw[-latex, blue] (48b)--(24b);
    \draw[-latex, blue] (47b)--(26b);
    \draw[-latex, blue] (46b)--(25b);
    \draw[-latex, blue] (45b)--(24b);
    \draw[-latex, blue] (49b)--(25b);	
    \draw[-latex, blue] (510)--(410);	
    \draw[-latex, blue] (24)--(05);
    \draw[-latex, blue] (25)--(05);
    \draw[-latex, blue] (26)--(05);
    \draw[-latex, blue] (59)--(49);
    \draw[-latex, blue] (58)--(48);
    \draw[-latex, blue] (57)--(47);
    \draw[-latex, blue] (56)--(46);
    \draw[-latex, blue] (55)--(45);
\end{tikzpicture}
\caption{Since $f\in V_1\Lambda^{e_1}$, we have $r(\Ff^{lO_{1}}(f)) = r(f)$
for any $l\in \NN$. Indeed, $r(\Ff^{lO_{0}}(f)) = r(f)$ for any $l\in \NN$.}
\label{fig1}
\end{center}
\end{figure}

We now recall the definition of the groupoid $G$ associated to a rank-2 Bratteli diagram
$\Lambda$ as defined in \cite{KP}. Denote by $\Omega_k$ the $k$-graph with vertices
$\Omega^0_k:=\NN^k$, paths  $\Omega^m_k:=\{(n,n+m), n\in \NN^k\}$ for $m\in \NN^k$,
$r((n,n+m))=n$ and $s((n,n+m))=n+m$. The \emph{infinite paths} in a $k$-graph $\Lambda$
with no sources are degree-preserving functors $x\colon \Omega_k\to \Lambda$. The
collection of all infinite paths of $\Lambda$ is denoted $\Lambda^\infty$, and we write
$r(x)$ for the vertex $x(0)$ and call it the the range of $x$. For $p\in \NN^k$ and $x\in
\Lambda^\infty$, $\sigma^p(x)\in \Lambda^\infty$ is defined by
$\sigma^p(x)(m,n):=x(m+p,n+p)$ \cite[Definition~2.1]{KP}. For $\lambda\in \Lambda$ and
$x\in \Lambda^\infty$ with $s(\lambda)=r(x)$ we write $\lambda x$ for the unique element
$y\in \Lambda^\infty$ such that $\lambda=y(0,d(\lambda))$ and $x=\sigma^{d(\lambda)}y$.
Then, as a set,
\[
    G=\{(x,n,y)\in \Lambda^\infty\times \ZZ^2 \times \Lambda^\infty \colon \sigma^{l}(x)=\sigma^{m}(y), n=l-m\}.
\]
Composition and inverse are given by $(x,n,y)(y,l,z)=(x,n+l,z)$ and
$(x,n,y)^{-1}=(y,-n,x)$, implying $r(x,n,y)=x$, $s(x,n,y)=y$. The topology on the
groupoid $G$ has basic open sets $Z(\lambda,\mu):=\{(\lambda{z},d(\lambda)-d(\mu),
\mu{z}) \colon z\in s(\lambda)\Lambda^\infty\}$, indexed by pairs $\lambda,\mu\in
\Lambda$ with $s(\lambda)=s(\mu)$. We call the topological groupoid $G$ the \emph{path
groupoid} of $\Lambda$. It is second countable, amenable, locally compact, Hausdorff and
ample.

\begin{lem}\label{lem8.3}
Let $\Lambda$ be a rank-2 Bratteli diagram. There is a unique automorphism $\alpha$ of
$\Lambda$ such that $\alpha(e) = \Ff^{m_n}(e)$ for all $e\in V_n\Lambda^{e_1}$. This
$\alpha$ induces a homeomorphism, also denoted $\alpha$, of $\Lambda^{\infty}$ by
$\alpha(x):=\alpha\circ x$, which in turn induces an automorphism $\alpha$ of $G$ such
that \[\alpha((x,m,y)) = (\alpha(x), m, \alpha(y)).\]
\end{lem}
\begin{proof}
There is at most one automorphism with the desired property because each $v\Lambda^{e_2}$
is a singleton and $\Lambda$ is generated as a category by $\Lambda^{e_1}\cup
\Lambda^{e_2}$. To show that there exist such an $\alpha\colon \Lambda\to \Lambda$, first
we consider $e,f \in \Lambda^{e_1}$ with $s(e)=r(f)$. We show
$s(\alpha(e))=r(\alpha(f))$. Let $n$ denote the level of $r(e)$. Then $r(f)$ is on level
$n+1$. Since $\Ff$ is a morphism and $\Ff^{O_n}(e)=e$ we get
$$r(\alpha(f))=r(\Ff^{m_{n+1}}(f))=s(\Ff^{m_{n+1}}(e))=s(\Ff^{m_{n} + nO_{n}}(e))
=s(\Ff^{m_n}(e))=s(\alpha(e)).$$

Thus we can extend $\alpha$ edgewise so that it is defined on any blue path. Because
$\Lambda$ is a rank-2 Bratelli diagram, $\alpha$ extends to an automorphism of $\Lambda$
as follows:  for $\lambda \in \Lambda$, factor $\lambda$ so that $\lambda =
\lambda_1\lambda_2$ where $\lambda_1$ is blue and $\lambda_2$ is red. Then
$\alpha(\lambda_1)$ makes sense and we define  $\alpha(\lambda_2)$ to be the unique red
path with degree $d(\lambda_2)$ and range equal to $s(\alpha(\lambda_1))$.  Now it is
routine to check that the formula $\alpha(\lambda) := \alpha(\lambda_1)\alpha(\lambda_2)$
is bijective and preserves composition and hence defines an automorphism of $\Lambda$.

Now $\alpha$ induces a homeomorphism $\alpha:\Lambda^{\infty} \to \Lambda^{\infty}$ by
$\alpha(x):=\alpha\circ x$. One checks that $\sigma^m\circ \alpha=\alpha\circ \sigma^m$
for all $m\in\NN^k$, and so there is an algebraic automorphism $\alpha$ of $G$ given by
\[
    \alpha((x,m,y)) = (\alpha(x), m, \alpha(y)).
\]
Since $\alpha:\Lambda^{\infty} \to \Lambda^{\infty}$ carries each $Z(\lambda,\mu)$
bijectively onto $Z(\alpha(\lambda), \alpha(\mu))$, $\alpha$ is a homeomorphism, and
hence an automorphism of the topological groupoid $G$.
\end{proof}

\begin{lem}\label{lem8.4}
Let $\Lambda$ be a rank-2 Bratteli diagram. Suppose that $o(e)>n\cdot m_n$ for every
$n\in\NN$ and $e\in V_n\Lambda^{e_1}$. Then the automorphism $\alpha$ of $G$ given by
Lemma~\ref{lem8.3} satisfies conditions \eqref{wfc}~and~\eqref{cond:lc}, and $G$ is
principal and minimal.
\end{lem}
\begin{proof}
To establish~\eqref{wfc} and prove that $G$ is principal, fix $x \in \Lambda^\infty$ and
$l\in\ZZ$ such that $[x] = [\alpha^l(x)]$. We claim that $l=0$. The shift-tail
equivalence class of $x$ is given by
\[
[x]:=\{\lambda\sigma^p(x)\colon p\in \NN^2, \lambda\in \Lambda x(p)\}.
\]
In particular $\sigma^p(x)=\sigma^q(\alpha^l(x))$ for some $p,q\in \NN^2$. Let $n$ denote
the level of $r(x)$. The map $\sigma^p$ increases the level of the range of an infinite
path by $p_1$, so $r(\sigma^p(x)) \in V_{n+p_1}$. On the other hand $\alpha$ does not
change the level, so $r(\sigma^q(\alpha^l(x))) \in V_{n+q_1}$, ensuring $p_1=q_1$.
Without loss of generality we may assume that $q_2\geq p_2$ (if not, replace $l$ by $-l$ and
interchange $p,q$). Define $y=\sigma^p(x)$ and $s=q_2-p_2$. Then
$y=\sigma^p(x)=\sigma^{q-p+p}(\alpha^l(x))=\sigma^{(0,s)}\alpha^l(y)$. Now observe that
for $r\in \NN$, $y(re_1,re_1+e_1+e_2)=ef=f'e'$ for some $e,e'\in\Lambda^{e_1}$,
$f,f'\in\Lambda^{e_2}$. It follows that $\sigma^{e_2}y(re_1,re_1 + e_1)=e'$, while
$\Ff(e')=y(re_1,re_1+e_1)$. So for $t\geq n+p_1$, putting $r_t:=(t-n-p_1)e_1$ and
$f_t:=y(r_t,r_t+e_1)$, we have
\[
f_t = y(r_t,r_t+e_1) =(\sigma^{(0,s)}\alpha^l(y))(r_t,r_t+e_1)=\alpha^l(y(r_t,r_t+e_1+se_2)).
\]
Since $y(r_t) \in V_t$, and by definition of $\alpha$, we have
$\alpha^l(y(r_t,r_t+e_1+se_2)) = \Ff^{lm_t}(y(r_t,r_t+e_1+se_2))$, and this in turn is
equal to $\Ff^{lm_t-s}(y(r_t,r_t+e_1))$ because $\Ff^{-s}(y(r_t,r_t+e_1)) =
\sigma^{(0,s)}y(r_t,r_t+e_1)$. Putting all this together, we obtain $f_t = \Ff^{lm_t -
s}(f_t)$.

Hence, for each $t > n + p_1$ we have $lm_t - s  = 0~\mathrm{mod}~o(f_t)$. By assumption,
we have $t\cdot m_t < o(f_t)$ for all $t$, and hence $|lm_t - s| < o(f_t)$ for large $t$.
This forces $lm_t - s=0$ for large $t$. Since $m_t \to \infty$, this forces $l=0$
establishing~\eqref{wfc}. It also forces $s=0$ and hence $p_2=q_2$. So if
$\sigma^p(x)=\sigma^q(x)$ then $p=q$, and this implies that $(x,m,x)\in G$ only for
$m=0$, so $G$ is principal.

To establish~\eqref{cond:lc}, for $\lambda\in\Lambda$ define $Z(\lambda):=\{\lambda{z}
\colon z\in s(\lambda)\Lambda^\infty\}$. Recall that $\Lambda^\infty$ is endowed with the
topology generated by the collection $\Bb:=\{Z(\lambda)\colon \lambda\in \Lambda\}$. Take
any $\lambda \in \Lambda$. We claim that $\alpha^{-l}(Z(\lambda)) \subseteq Z(\lambda)$
for some $l\in\NN\setminus\{0\}$. To see this, let $n,i$ be the pair with $r(\lambda)\in
V_{n,i}$. Since $\Ff$ permutes the elements of each $V_{n,j}$, so does $\alpha$ and hence
each $\alpha^l$ is a permutation of the finite set $V_{n,i}\Lambda^{d(\lambda)}$. In
particular $\alpha^l(\lambda)=\lambda$ for some $l>0$. Hence for each $x\in Z(\lambda)$,
with $x=\lambda{z}$, we have $\alpha^{-l}(x) = \alpha^{-l}(\lambda{z}) =
\alpha^{-l}(\lambda)\alpha^{-l}(z) = \lambda\alpha^{-l}(z)\in Z(\lambda)$.

It remains to check that $G$ is minimal, which follows from \cite[Theorem~5.1]{BCFS}
because $\cs(G)$ is simple by \cite[Theorem~5.1]{PRRS}.
\end{proof}

\begin{proof}[Proof of Theorem~\ref{thm8.1}]
Using Lemma~\ref{lem8.2}, we construct a rank-2 Bratteli diagram $\Lambda$ with
$\Lambda^0=\bigcup_{n=0}^\infty V_{n}$ such that $K_0(C^*(\Lambda))\cong \lim (\ZZ^{c_n},
A_n)$, $K_1(C^*(\Lambda))\cong \lim (\ZZ^{c_n}, B_n)$ and $o(e)>n\cdot m_n$ for every
$n\in\NN$ and $e\in V_n\Lambda^{e_1}$.

Let $G$ denote the path groupoid of $\Lambda$. By construction $G$ is a second countable,
amenable, locally compact, Hausdorff, ample groupoid and $\cs(\Lambda) \cong \cs(G)$
\cite{KP}. By Lemmas \ref{lem8.3}~and~\ref{lem8.4}, there is an automorphism $\alpha$ of
$G$ that satisfies conditions \eqref{wfc}~and~\eqref{cond:lc}, and $G$ is principal and
minimal, $K_*(\iota_{\widetilde{G}})$ is an isomorphism, and $\cs(\tginf)$ is a Kirchberg
algebra in the UCT class.

We now proceed as in the proof of Theorem~\ref{thm:AFeg}. Lemma~\ref{lem:stabilise} shows
that $\widetilde{G} := G \times K$ and $\widetilde{\alpha} := \alpha \times \id$ have all
the properties established for $G$ and $\alpha$ in the preceding paragraphs. We apply
Theorem~\ref{thm:main} to see that $\tginf$ is second-countable, amenable, locally
compact, Hausdorff, ample groupoid that is principal and minimal. Moreover,
$\cs(\widetilde{G})$ is nonunital because $\widetilde{G}^{(0)} = \go\times \ZZ$ is
noncompact.

First suppose that $A$ is nonunital. Then $A$ and $\cs(\tginf)$ are nonunital Kirchberg
algebras with the same $K$-theory:
\[
K_0(A) \cong \varinjlim (\ZZ^{c_n}, A_n)\cong K_0(C^*(\Lambda))\cong K_0(C^*(G)) \cong K_0(C^*(\widetilde{G}))\cong K_0(C^*(\tginf)),
\]
and a similar calculation gives $K_1(A) \cong K_1(C^*(\tginf))$. Hence $A\cong
C^*(\tginf)$ by the Kirchberg--Phillips theorem \cite{Kirchberg, Phillips}.

Now suppose that $A$ is unital. Then by assumption, the isomorphism $K_0(A) \cong
\varinjlim (\ZZ^{c_n}, A_n)$ carries $[1_A]$ to an element of the positive cone, and
hence the class of some $a \in \NN^{c_n} \subseteq \ZZ^{c_n}$. For each $j \le c_n$ chose
$v_j \in V_{n,j}$. As discussed on page~\pageref{pg:isodesc} above, Theorem~6.2(2) of
\cite{PRRS} shows that the isomorphism $\lim (\ZZ^{c_n}, A_n) \cong K_0(\cs(G))$ takes
$a$ to $\sum_{j \le c_n} a(j)[1_{Z(v_j)}]$. So the isomorphism $\varinjlim (\ZZ^{c_n},
A_n) \cong K_0(\cs(\widetilde{G}))$ carries $a$ to $\sum_{j \le c_n} \sum_{k \le a(j)}
[1_{Z(v_j) \times \{k\}}]$, which is the class of a characteristic function $1_V$ of a
clopen subset $V$ of $\widetilde{G}^{(0)}$. So the argument of the final few paragraphs
of the proof of Theorem~\ref{thm:AFeg} shows that $\Gg := W \tginf W$ with $W:=\hio
\times V$ has the desired properties.
\end{proof}

\end{document}